\documentclass[final,11pt]{article}

\usepackage{latexsym}
  \usepackage{multicol}
\usepackage{fancybox}
     \usepackage{xcolor}
\usepackage{amsmath}
\usepackage{cite}
\usepackage{amsmath}
\usepackage{amssymb}
\usepackage{amsthm}
\usepackage{enumerate}
\usepackage{xcolor} %

\newtheorem{theorem}{Theorem}

\newtheorem{remark}{Remark}

\newtheorem{example}{Example}[section]
\renewcommand{\baselinestretch}{1.1}
\def\eqsp{\noalign{\vskip 5pt}}
\def\dsp{\displaystyle}

\usepackage{bm}
\usepackage{graphicx}
\usepackage{framed}

\usepackage[all]{xy}

\newcommand{\gm}[1]{\gamma_{#1}}

\newcommand{\lp}{\left (}
\newcommand{\rp}{\right )}

\newcommand{\eqn}[1]{(\ref{#1})}

\newcommand{\eq}{\begin{equation}}
\newcommand{\en}{\end{equation}}
\newcommand{\eqm}{\begin{eqnarray}}
\newcommand{\enm}{\end{eqnarray}}
\newcommand{\eqmno}{\begin{eqnarray*}}
\newcommand{\enmno}{\end{eqnarray*}}
\newcommand{\eql}[1]{\begin{equation}\label{#1}}

\newcommand{\fr}[2]{{\frac{#1}{#2}} }

\newcommand{\vthin}{\vskip 5pt}

\newcommand{\ignore}[1]{}
\def\bc{\begin{center}}
\def\ec{\end{center}}
\def\bi{\begin{itemize}}
\def\ei{\end{itemize}}
\def\be{\begin{enumerate}}
\def\ee{\end{enumerate}}

\def\alf{\alpha}

\def\reals{{{\rm l} \kern -.15em {\rm R} }}

\def\qquad{\quad\quad}

\def\eqsp{\noalign{\vskip 5pt}}
\def\qquad{\quad\quad}

\def\dsp{\displaystyle}
\newcommand{\eqml}[1]{\eql{#1}\begin{array}{rcl}}
\newcommand{\enml}{\end{array}\end{equation}}

\usepackage{multicol}
\usepackage{fancybox}
     \usepackage{xcolor}
  \usepackage{showlabels}
  \usepackage{hyperref}
\usepackage{cite}
 \usepackage{showkeys}

\usepackage{cases}
\usepackage{multirow}
\usepackage{multicol}
\usepackage{arydshln}
\usepackage{amsmath}
\usepackage{amssymb}

\newcommand{\bfalf}{\mbox{\boldmath $\alpha$}}
\newcommand{\bfbeta}{\mbox{\boldmath $\beta$}}

\renewcommand{\baselinestretch}{1.1}

\usepackage{bm}
\usepackage{graphicx}
\usepackage{framed}
\usepackage{epstopdf}

\usepackage[all]{xy}

\begin{document}

\title{High order two-grid finite difference methods for interface and internal layer problems}

\author{Zhilin Li \thanks{Department of Mathematics, North Carolina State University, Raleigh,
NC 27695-8205, USA, Email: zhilin@math.ncsu.edu}
\and
Kejia Pan\thanks{School of Mathematics and Statistics, HNP-LAMA, Central South University, Changsha, Hunan, 410083, China, Email: kejiapan@csu.edu.cn }
\and Juan Ruiz-\'Alvarez \thanks{Departamento de Matem\'atica  Aplicada y Estad\'istica. Universidad  Polit\'ecnica de Cartagena,  Spain, Email: juan.ruiz@upct.es}
}

\date{}
\maketitle


\begin{abstract}

Second order accurate Cartesian grid methods have been well developed for interface problems in the literature. However,
it is challenging to develop third or higher order accurate methods for problems with curved interfaces and internal boundaries.
 In this paper, alternative approaches based on two different grids are developed for some interface and internal layer problems,
 which are  different from adaptive mesh refinement (AMR) techniques. 
 For one dimensional,  or two-dimensional problems with straight interfaces or boundary layers that are parallel to one of the axes,  the discussion is relatively easy.  One of  challenges is how to construct
 a fourth order compact finite difference scheme  at boarder grid points that connect two meshes.
 A two-grid method that employs a second order discretization near the interface in the fine mesh and a fourth order discretization away from the interface in the coarse  and boarder grid points is proposed.  
 For two dimensional problems with a curved interface or an internal layer,  a level set representation is utilized for which we can build a fine mesh within a tube $|\varphi({\bf x}) | \le \delta h$ of the interface. A new super-third  seven-point discretization that can guarantee the discrete maximum principle has been developed at hanging nodes.  The coefficient matrices of the finite difference equations developed in this paper are M-matrices,  which leads to the convergence of the finite difference schemes. Non-trivial numerical examples presented in this paper have confirmed the desired accuracy and convergence.

 \end{abstract}

{\bf keywords:}
 Two-grid method,  interface problems,   internal layer problems, finite difference method,  discrete maximum principle, hanging nodes.


\section{Introduction}

\smallskip

High order compact schemes are useful or even required for many applications such as high wave numbers, oscillatory solutions, and problems on large or infinite domains. In general,
high order compact schemes are suitable for rectangular meshes and boundaries.
For interface problems, the solutions often depend on the curvature of the interface, which presents challenges to develp  high order  discretizations near the interface. 
Various efforts have been made recently. For singular source only interface problems, fourth order schemes  for Poisson equations have been developed, see for example \cite{Xie-Ying-20,pan-he-li-HOC21};  and a sixth order method if the interface has an analytic expressions \cite{sixth_IIM}.  A third order scheme for elliptic partial differential equations (PDEs)  with piecewise constant but discontinuous coefficients  in the presence of fixed interfaces in \cite{pan-he-li-HOC21}. The implementation for high order schemes for interface problems with curved interface is rather complicated.
Note that the curvature of an interface  involves second order derivatives, and  is non-linear.  A related topic is how to resolve the solution accurately near an internal layer, which is a long outstanding challenging problem in computational fluid dynamics and other applications especially if the boundary layer is not parallel to one of the axes.  It is well-known that often the accuracy of the  pressure is one order lower than  that of the velocity due to the boundary layer effect in solving incompressible Naiver-Stokes equations \cite{chorin:proj2,brown-cortez-minion,cortez-varela}.
In some sense, to resolve the solution near an internal (boundary) layer is like solving an interface problem in a structured mesh. 

Our  idea in this paper is to use  two-grid methods to solve these challenging problems. The idea is different from adaptive mesh refinement (AMR) techniques \cite{MR3026197,MR3091860,Berger-Colella,Griffith-Hornung-McQueen-Peskin,Griffith,li_song_AMR,pengs}.
For one dimensional (1D) interface problems,
we use a fourth order method on a mesh with  a step size $h$ away from the interface; and a  second order method with a finer mesh  with step size $h/r$  near the interface. Thus, we can get an approximation with  $O(h^4) + O((h/r)^2)$ global accuracy. In the left diagram in Figure~\ref{fig:1dgrid},  we show a diagram of  a two-grid for an interface problem (top); and a boundary layer problem (bottom). In the right  plots in Figure~\ref{fig:1dgrid}, we show  the solution and error plots  for a famous boundary layer problem.  In Figure~\ref{tab:2D_1D}, we show the solution and a two-grid for a  two-dimensional example with a line interface;  and in  Figure~\ref{fig:amr_mesh}  a two-grid for a circular interface.
 Interested readers can get an idea of the proposed two-grid methods from those plots.

Note also that there are different interpretations of `two-grid methods' in the literature including different  meshes for different variables  \cite{xu-two-grid-96,MR4309805,MR1317859}; different meshes for interior and boundary domains \cite{MR4027793,MR1317859,tw-grid-bdl,MR4291187}. For boundary layers that are parallel to one of the axes,  a Shishkin mesh has been employed to capture the boundary layer effect, see for example, \cite{Shishkin-mesh,Shishkin-mesh-95,Shishkin-mesh-10,MR3284632}.

In this paper, we develop high order two-grid methods that are as compact as possible with some new ideas. One of the keys in our methods is that we try to enforce the sign properties for  elliptic differential equations  so that the discrete maximum principle can be preserved, which leads to the convergence of the finite difference methods. For one-dimensional problems, we have developed a fourth-order compact scheme for Poison/Helmholtz equations with non-uniform grids. For two-dimensional (2D) interface problems with a line interface that is parallel to one of the axes, we have developed a new finite difference method that is fourth order away from the interface,  second order in one coordinate direction with a fine mesh,  and fourth order in the other with a coarse mesh.
For 2D general interface problems, the algorithm and discussion  are more complicated. One of difficulties is how to develop high order discretizations at hanging nodes, which can preserve  the discrete maximum principle. Different from a traditional approach using ghost points and interpolations, which cannot guarantee the stability of the algorithm, we use a new undetermined coefficients method so that the sign property can be satisfied with the highest possible accuracy. The convergence of the developed two-grid methods is guaranteed with the sign property and the local truncation error analysis for elliptic interface problems and internal layer problems.

The rest of paper is organized as follows. In the next section, we discuss the two-grid method for one-dimensional interface problems, and two-dimensional interface problems with line interfaces that are parallel to one of the axes. Numerical examples are also presented.  In Section~\ref{sec:2D}, we discuss the two-grid method for general elliptic interface problems,   some implementation details, and  a new finite difference discretization at hanging nodes, which can preserve the discrete maximum principle. Several benchmark examples from the literature are tested and analyzed. We also present a numerical experiment of a non-trivial internal layer problem. We conclude in the last section.

\section{Two-grid methods for 1D interface problems and 2D problems with line interfaces}

We first show a simple and well-known one-dimensional boundary layer example,
\eqm
  \epsilon u'' -  u'  +   u = f(x), \quad 0< x < 1, \qquad u(0) =1, \quad u(1) = 3.
\enm
The solution  exhibits a boundary layer near $x=1$ if $\epsilon$ is small.

\begin{figure}[htbp]
\centering
\includegraphics[height=0.4\textwidth]{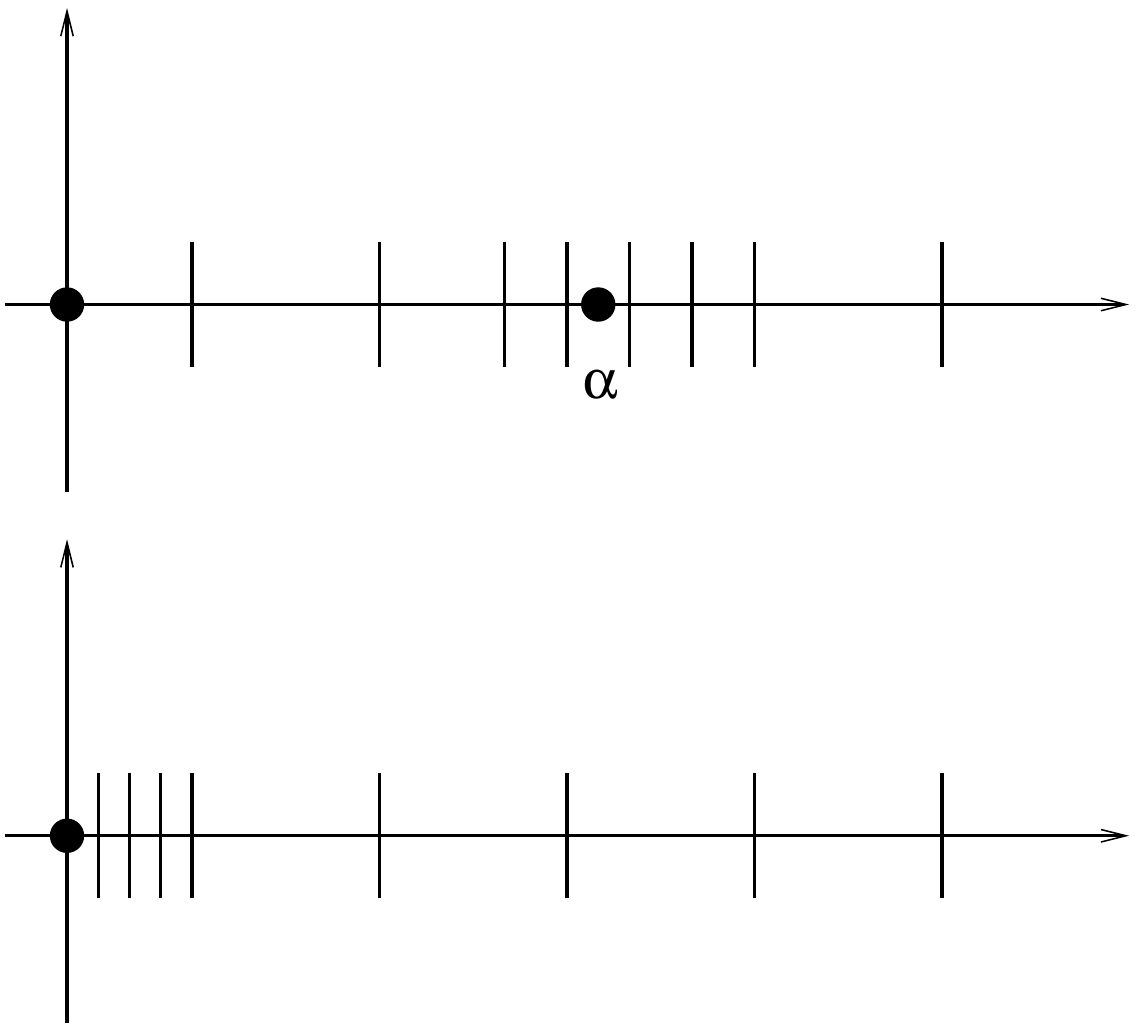}
\hfil
  \includegraphics[height=0.4\textwidth]{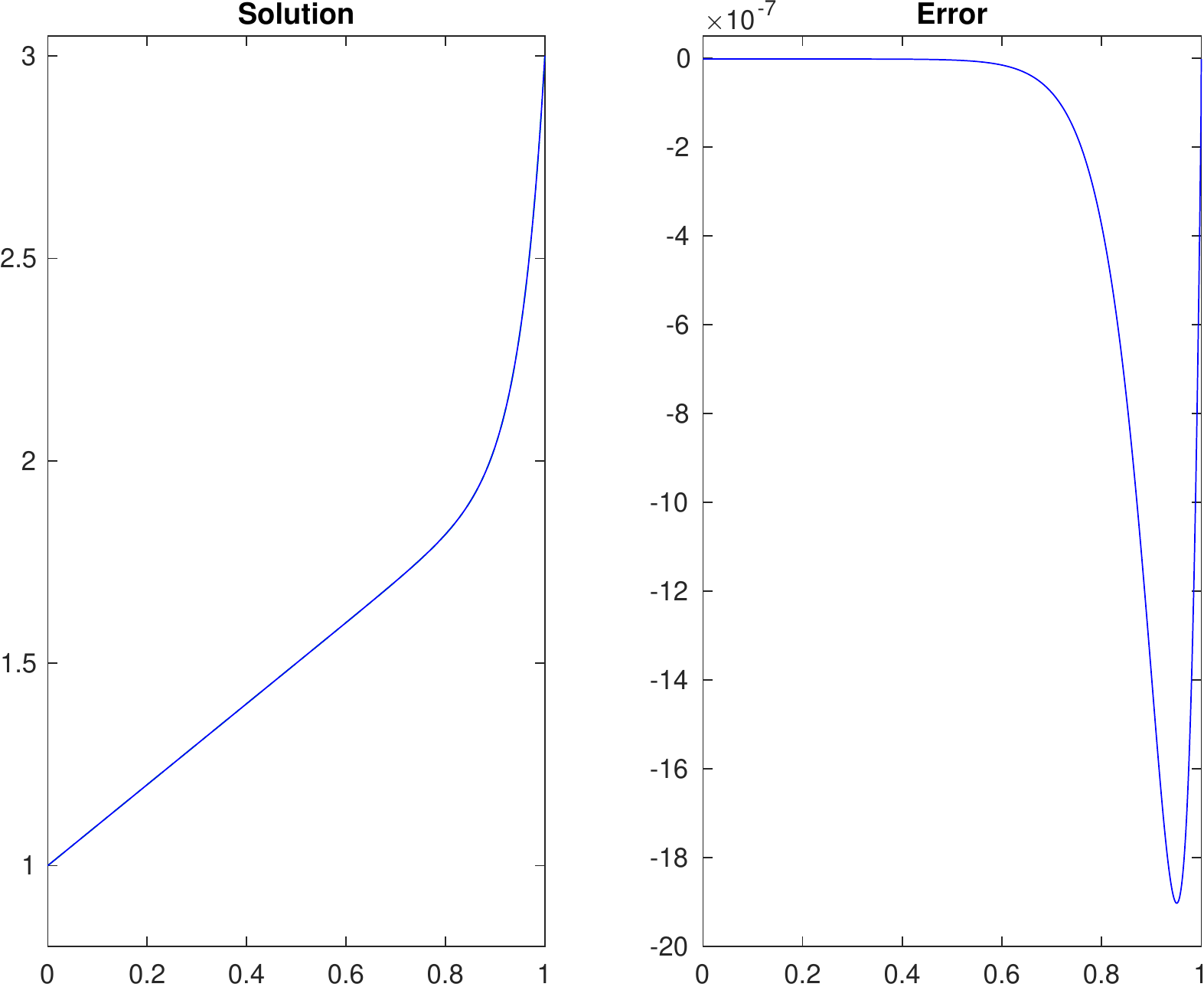}
\caption{Diagrams of  the two-grids method in one-dimension.  Left-top: for an interface problem; Left-bottom: for a boundary layer problem.  Right plots: A computed boundary layer solution and the error ($\sim 10^{-6}$) with a coarse mesh size $h=1/N=1/80$, and a fine mesh $h_f=h^2$ in $(1-2h, 1)$.  } \label{fig:1dgrid}
\end{figure}

For this example, we show the computed solution and error plot in Figure~\ref{fig:1dgrid} using a standard centered finite difference discretization. Let $h$ be a coarse grid  size and $N$ be the number of grid points in the interval $(0,\,1)$.  In the interval $[1-2h, 1]$, where $h=1/N$, we use a fine mesh $h_f=h^2$ for the discretization so that the central discretization has little effect on the solution even though that $|u'|$ is relatively large.

Next, we  use one-dimensional interface problems,
\eqm
  (\kappa \, u')'  +  K u = f(x) + C\delta(x-\alf) + \bar{C}\delta'(x-\alf) , \qquad   a < x < b,\quad a<\alf <b,
\enm
with two linear boundary conditions (Dirichlet, Neumann, Robin) at $x=a$ and $x=b$ to demonstrate the two-grid method. The problem is equivalent to
\eqm
  (\kappa  u')' +  K u = f(x), \qquad  [u]_{\alf} = \frac{2\bar{C}}{\kappa^- + \kappa^+ }, \quad [\kappa  u']_{\alf} = C, \quad a < x < b,
\enm
see \cite{REU10}. We assume that the coefficient $\beta$ is a piecewise constant,
\eqm
 \kappa =  \left \{ \begin{array}{ll}
\kappa^-    &   \mbox{if   $a< x<\alf$ } \\ \eqsp
  \kappa^+   &    \mbox{ if   $\alf < x <b$. }
  \end{array} \right.
\enm


Given a coarse grid parameter $N$, a grid ratio $r$,  and a width parameter $\lambda>0$, we can easily generate a two-grid. First, we define $h=(b-a)/N$,  and let the fine mesh size be $h_f = h/r$. The fine mesh is defined as $x_j= a + j h_f$   if $|x_j - \alf | \le  \lambda  h$. The coarse mesh is defined as those $x_i = a + h i$ and $|x_i - \alf | > \lambda  h$, see the left plot in Figure~\ref{tab:1D_E2_2} for an illustration.

The standard  fourth order compact scheme at a grid point $x_i$ for a uniform mesh, {\em i.e.} $x_{i+1} - x_i= x_i - x_{i-1}=h$,  is
\eqm
 \kappa  \frac{U_{i-1} - 2 U_i + U_{i+1} } { h^2}  = \frac{f_{i-1} + 10 f_i + f_{i+1}  } { 12}.
\enm

\subsection{The IIM scheme near the interface} \label{sec-IIM}

While it is possible to have fourth order schemes for regular one dimensional problems,  see for example, Chapter 7 in \cite{li:book}, it is difficult to obtain fourth order accurate schemes in two and three dimensions for discontinuous coefficients and curved interfaces.  One of the purposes of the two-grid method is to    
provide an alternative approach that can obtain comparable high order accuracy near the interface. We use a second order method near the interface with a finer mesh but a fourth order scheme away from the interface.

\begin{figure}[hpbt]
\bc
 \includegraphics[width=0.4\textwidth]{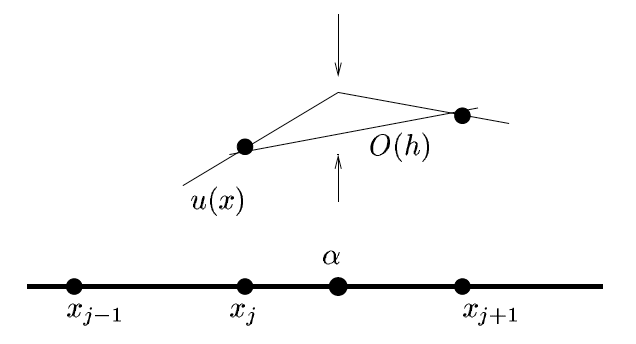}
\caption{A diagram of some  grid points and the interface $\alf$ that is
between $x_j$ and $x_{j+1}$.}
\label{1d_grid}
\ec
\end{figure}

Assume that the interface $\alf$ is between two grid points $x_j$ and $x_{j+1}$, that is, $ \le \alf < x_{j+1}$.  The interface is  one of the grid points when $x_j = \alf$, see Figure~\ref{1d_grid} for an illustration. The finite difference  equations at two irregular grid points $x_j$ and  $x_{j+1}$  have the following forms, see \cite{li:book},
\eqml{eq1d2.6}
&  \dsp \gamma _{j,1} U_{j-1} + \gamma _{j,2} U_{j}+ \gamma _{j,3} U_{j+1}
 + K U_{j}= f_{j} + C_j, \\  \eqsp
& \dsp \gamma _{j+1,1} U_{j} + \gamma _{j+1,2} U_{j+1}+ \gamma _{j+1,3}
U_{j+2} +K U_{j+1}
 = f_{j+1} + C_{j+1}.
\enml
The coefficients of the finite difference  equations  for the special case $\bar{C}=0$ have the following closed form:
\eqml{1dfd} &&   \left \{ \begin{array}{l}
  \dsp \gm{j,1} = (\kappa^- - [\kappa] {(x_{j} -\alf)}/{h_f} ) / D_j, \\ \eqsp
  \dsp  \gm{j,2} = (-2 \kappa^-+ [\kappa] {(x_{j-1} -\alf)}/{h_f} ) /D_j, \\ \eqsp
  \dsp \gm{j,3} = {\kappa^+ } /D_j, \end{array} \right. \\ \eqsp \eqsp
 &&   \left \{ \begin{array}{l}
 \dsp \gm{j+1,1} = {\kappa^-}/D_{j+1},  \\ \eqsp
 \dsp \gm{j+1,2} = (-2 \kappa^+ + [\kappa] {(x_{j+2} -\alf)}/h_f )/D_{j+1}, \\ \eqsp
 \dsp \gm{j+1,3} = (\kappa^+ - [\kappa] {(x_{j+1} -\alf)}{}/h_f )/D_{j+1},
\end{array} \right.
\enml
where
\eqml{Dj_Djp1}
D_j &=& \dsp h_f^2 + [\kappa](x_{j- 1}-\alf)(x_j-\alf)/2\kappa^-,\\ \eqsp
D_{j+1} &=& \dsp h_f^2 - [\kappa](x_{j+2}-\alf)(x_{j+1}-\alf)/2\kappa^+.
\enml
The correction terms are given by
\eql{eq1d2.22}
C_j = \gm{j,3}\, (x_{j+1}-\alpha)\,\fr{C}{\kappa^+}, \qquad
C_{j+1} = \gm{j+1,1}\,  (\alpha -x_{j})\,\fr{C}{\kappa^-}.
\en
The  stability and consistency of the finite difference scheme are discussed in \cite{li:book}.

\subsection{High-order compact FD discretization at boarder grid points}\label{subsec}

One of challenges in two-grid methods is to design a compatible high order finite difference  discretization at grid points that border the coarse and fine grids.
In this subsection, we assume that $\kappa$ and $K$ are constants.
 Let $x_i$ be such a grid. Assume that $x_i - x_{i-1} = h_1$ and $x_{i+1}-x_i= h_2$.
We design the following compact finite difference scheme,
\eqm
  \sum_{k=-1}^1 \alf_k U_{i+k} =  \sum_{k=-1}^1 \beta_k f_{i+k}, \qquad   \sum_{k=-1}^1 \beta_k = 1.
\enm
There are six unknowns, $ \alf_k $ and $\beta_k$, $k=-1,0, \, 1$. Define the `local truncation error' as
\eqm
  T_i = \sum_{k=-1}^1 \alf_k u(x_{i+k}) -  \sum_{k=-1}^1 \beta_k f(x_{i+k}),
\enm
assuming that $u(x)$ is the true solution. We want the finite difference scheme  to be exact for all fourth order polynomials if possible. To achieve this, we expand the right hand side above at $x_i$ to get
\eqmno
  T_i = \sum_{k=-1}^1 \alf_k \sum_{l=0}^4  \frac{u^{(l)} (x_i) }  {l !}  \hat{h}_k^l  -  \sum_{k=-1}^1 \beta_k \sum_{l=0}^2  \frac{f^{(l)}(x_i)}  {l !}  \hat{h}_k^l  + O\left(\| {\bfalf} \|_{\infty} h^5 + \| {\bfbeta} \|_{\infty} h^3 \right ),
\enmno
where $\hat{h}_{-1}=-h_1$, $\hat{h}_{0}=0$,  $\hat{h}_{1}=h_2$,  $h= \max\{h_1,h_2\}$, and $\| {\bfalf} \|_{\infty} = \max\left\{|\alpha_{i}|\right\}$ and so on. Note that the first and second order derivatives of $f$ can be expressed as the derivatives of the solution below,
\eqm
  f^{(1)}(x_i) = \kappa  u^{(3)} (x_i) + K u'(x_i) , \qquad  f^{(2)}(x_i) = \kappa  u^{(4)} (x_i) + K u''(x_i) ,
\enm
from the differential equation. Thus, by collecting corresponding terms involving $u^{(l)}$, $0\le l \le 4$, we get a system of equations $A x=b$ for the coefficients $ \alf_k $ and $\beta_k$, $k=-1,0,1$. Note that the constraint $\sum \beta_k=1$ is necessary to avoid  the trivial solution.
Thus, in the 1D case, we have the same number of unknowns and constraints and we have  uniquely determined the coefficients.
We list the coefficient matrix and the right hand side below for the special case $\kappa=1$, $K=0$.
\eqm
A = \left(\begin{array}{cccccc} 1 & 1 & 1 & 0 & 0 & 0\\ \eqsp -h_1 & 0 & h_2 & 0  & 0 & 0  \\ \eqsp \frac{h_1^2}{2}  & 0 &  \frac{h_2^2}{2}  & -1  & -1 & -1   \\ \eqsp
-\frac{h_1^3}{6} & 0 & \frac{h_2^3}{6} &h_1  & 0 & -h_2\\ \eqsp  \frac{h_1^4}{24} & 0 &  \frac{h_2^4}{24} & -\frac{h_1^2}{2}  & 0 & -\frac{h_2^2}{2} \\ \eqsp 0 & 0 & 0 & 1 & 1 & 1 \end{array}\right), \quad b=\left(\begin{array}{c} 0\\ 0\\ 0\\ 0\\ 0\\ 1 \end{array}\right).
\enm
The solution (the set of coefficients) has an analytic form that is obtained from the Maple symbolic package,
\eqml{coef}
&& \dsp \alf_{-1} = \frac{2}{h_1(h_1+h_2)}, \qquad \alf_1 = \frac{2}{h_2(h_1+h_2)}, \qquad \alf_0 = - \lp \alf_l + \alf_r  \rp = -\frac{2}{ h_1 h_2}, \\ \eqsp \eqsp
&& \dsp \beta_{-1} = \frac{h_1^2-h_2^2 + h_1 h_2}{6h_1(h_1+h_2)}, \qquad \beta_1 = \frac{- h_1^2+h_2^2 + h_1 h_2}{ 6h_2(h_1+h_2) }, \qquad \beta_0 =    \frac{h_2^2+h_1^2 + 3h_1 h_2}{ 6h_1 h_2}.
\enml
We can see the `symmetry' between the left and right grid points as expected.
Note that when $h_1=h_2$ we  recover the standard fourth-order compact formula. The finite difference coefficients satisfy the sign property needed for an M-matrix conditions. When $K\neq 0$, we can treat $Ku$ as a source term like $f$, that is, we add the term
\eqm
  K(\beta_{-1} U_{i-1} +  \beta_0 U_i + \beta_1 U_{i+1})
\enm
to the left hand side of the finite difference equation.

\subsection{A one-dimensional numerical example} \label{sub1dex}

This example is from \cite{li:book}:
\eqmno
    && (\kappa u_x)_x = 12 x^2, \quad 0 <x <1, 
  \quad  \kappa = \left \{ \begin{array}{ll}
       \kappa^- & \mbox{if $x<\alf$} \\ \eqsp
        \kappa^+  & \mbox{if $x>\alf$},
        \end{array}
\right. \\ \eqsp
 &&    \dsp  u(0)=0, \quad u(1) =\frac{1}{\kappa^+} + \lp
 \frac{1}{\kappa^-}-\frac{1}{\kappa^+}\rp \alf^4.
\enmno
In this example, $f(x)=12 x^2$ is continuous and the
natural jump conditions \index{natural jump conditions}
$[u]_{\alf} =0$ and $[\kappa u_x]_{\alf} =0$ are satisfied
across the interface $\alf$. The exact solution is
\eqmno
u(x) = \left \{ \begin{array}{ll}
       \dsp \frac{x^4}{\kappa^-}, & \mbox{if $x<\alf$}, \\ \eqsp
      \dsp  \frac{x^4}{\kappa^+}, & \mbox{if $x>\alf$}.
        \end{array}
\right.
\enmno

\begin{figure}[htbp]
\begin{minipage}[c]{2.75in}
     \includegraphics[width=1.1\textwidth]{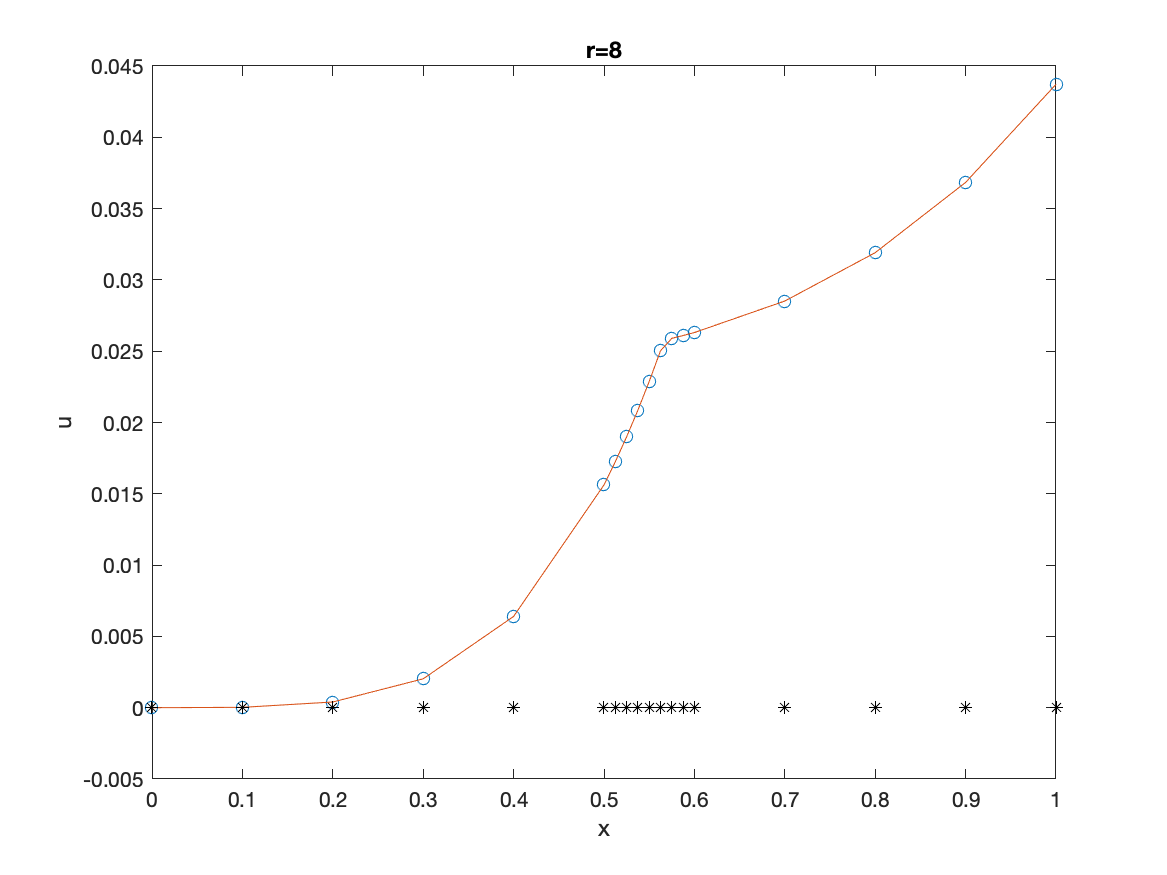}

{\small A two-grid along the $x$-axis,   the exact solution (solid line) and computed solution (little o's) with $N=10$, $\lambda=2$, and $r=8$.  The maximum error is $4.5343\times 10^{-7}$.}
\vthin

\end{minipage}
\begin{minipage}[t]{2.5in}
$\null$

\vspace{-3.65cm}
$\null \qquad$
\begin{tabular}{|c||c | c|}
\hline
     $N$ & $\|E \|_{\infty}$ (uniform)&    $\|E \|_{\infty}$ (two-grid) \\  \hline
           \multirow{4}{*}{10} &\multirow{4}{*}{$4.4333 \times 10^{-4} $}    & $ 7.286 1\times 10^{-5}$   \\ \cline{3-3}
      &  & $1.1428   \times 10^{-5}$     \\  \cline{3-3}
      &   &  $4.5343  \times 10^{-6}$    \\  \cline{3-3}
      &   &  $  6.1344 \times 10^{-7}$    \\ \hline
 \multirow{3}{*}{20} &\multirow{4}{*}{$ 1.3470 \times 10^{-4}$}  &    $ 1.0335  \times 10^{-5}$  \\  \cline{3-3}
    &  &  $4.5343 \times 10^{-6}$    \\  \cline{3-3}
   &   &  $  6.1344  \times 10^{-7}$   \\  \cline{3-3}
  &    &  $   2.8222 \times 10^{-7}$   \\   \hline
    \multirow{3}{*}{40} &\multirow{4}{*}{$ 4.1331 \times 10^{-5}$}  &   $  4.4768\times 10^{-6}$ \\  \cline{3-3}
  &   & $  6.1344 \times 10^{-7}$    \\  \cline{3-3}
   &  & $ 2.8222 \times 10^{-7}$    \\  \cline{3-3}
  &   & $   3.9592 \times 10^{-8}$     \\  \hline
     \end{tabular}
  \end{minipage}
  \caption{Computed results of the 1D example.} \label{tab:1D_E2_2}
 \end{figure}
\ignore{
\begin{minipage}[t]{1.0in}

\vspace{-2.5cm}
\begin{tabular}{|c||c |}
\hline
     $N$ &   Error \\  \hline
     \multirow{4}{*}{80}  & $  5.9917\times 10^{-7}$   \\  \cline{2-2}
          & $  2.7459 \times 10^{-7}$    \\  \cline{2-2}
     & $  3.9110 \times 10^{-8}$    \\  \cline{2-2}
     & $  2.4493 \times 10^{-8}$     \\  \hline
      \multirow{4}{*}{160}  & $  2.7189 \times 10^{-7}$   \\  \cline{2-2}
          & $ 3.8139 \times 10^{-8}$    \\  \cline{2-2}
     & $ 1.7667  \times 10^{-8}$    \\  \cline{2-2}
     & $  2.4493 \times 10^{-9}$     \\  \hline
 \multirow{4}{*}{320}  & $ 3.5959  \times 10^{-8}$   \\  \cline{2-2}
          & $  1.7398\times 10^{-8}$    \\  \cline{2-2}
     & $  2.4341 \times 10^{-9}$    \\  \cline{2-2}
     & $  1.1497 \times 10^{-9}$     \\  \hline

\end{tabular}
\end{minipage}   }
We use the  two-grid method to solve the  problem and show the numerical results in Figure~\ref{tab:1D_E2_2}. In the left plot, we can clearly visualize  the two grids in which the parameters are  $N=10, \alf=17/30, \lambda=2$, that is, we use a fine mesh in the interval $|x-\alf| \le 2 h$; with the refinement ratio $r=h/h_f=8$,  and $\kappa^-=4, \kappa^+=50$. The total number of grids points is $16$,  and the maximum error is $4.5343\times 10^{-7}$. In the right table, we carry out some grid refinement analysis with various coarse meshes and mesh ratio $r=2,4,8,16$. The average convergence order is $2.0875$ with respect to $h/r$ as expected. We observe better than fourth order convergence from $r=2$ to $r=8$. Note that the fine mesh is within a tube and thus,   it is more efficient in terms of the accuracy and complexity compared with that of a fourth high method using a uniform mesh.

\subsection{2D problems with line interfaces that are parallel to one of the axes}

Now we consider two dimensional Poisson equations with a line interface that is parallel to the $y$-axis
\eqm \label{2d-1d}
  u_{xx} + u_{yy} = f(x,y) + C \delta (x- \alpha) + \bar{C}  \delta' (x- \alpha),  \qquad a<x < b, \quad c<y<d,
\enm
where $C$ and $\bar{C}$ correspond to the jumps of $u$ and $\frac{\partial u}{\partial  x}$ across $\alpha$,  $[u]_{\alf} =  \bar{C}$ and $[u_x]_{\alf} =  {C}$. That is, all quantities in the $y$-direction are continuous.
More general problems with piecewise constant coefficients and curved interfaces will be discussed in the next section.

 In the literature, a two-grid in one space dimension is often referred to as the Shishkin mesh that has been employed to capture boundary layer effects, see for example, \cite{Shishkin-mesh,Shishkin-mesh-95,Shishkin-mesh-10,MR3284632}.
To solve \eqn{2d-1d}, we use a two-grid in the $x$-direction as discussed in the previous sub-section; and a uniform mesh in the $y$-direction. Then,
there are  three types of grid points: coarse grid points (black dotted), fine grid points
(blue dotted), and  boarder grid points (red dotted),  as demonstrated in Figure \ref{fig:2dmesh}.

For a coarse gird point $(x_i,y_j)$  that is outside of the strip $|x_i - \alf | > \lambda h$,  we use the standard compact
 nine-point finite difference scheme, 
see for example, \cite{strikwerda,li:book-FDFEM}.
The fourth order compact scheme for $u_{xx} + u_{yy} + KU= f(x,y)$  assuming the same step size in the $x$ and $y$ directions  can be written as
\eqm
\lp  L_h + K M_h\rp U_{ij}   = M_h f_{ij},
\enm
 where $L_h$ is the discrete nine-point  Laplacian operator whose coefficients at the four corners are $\frac{1}{6 h^2}$, at the
east-north-south-west grid points are $\frac{4}{6 h^2}$, and at the center is $-\frac{20}{\, 6 h^2}$; $M_h$ is an averaging  operator,
\eqm
  M_h f_{ij} = \frac{1}{12} \left ( \frac{\null}{\null} f_{i-1,j} + f_{i+1,j}+f_{i,j-1} + f_{i,j+1} + 8 f_{i,j}\right),
\enm
where $f_{ij} = f(x_i,y_j)$ and so on.

\begin{figure}[!tbp]
\centering
\includegraphics[width=0.4\textwidth]{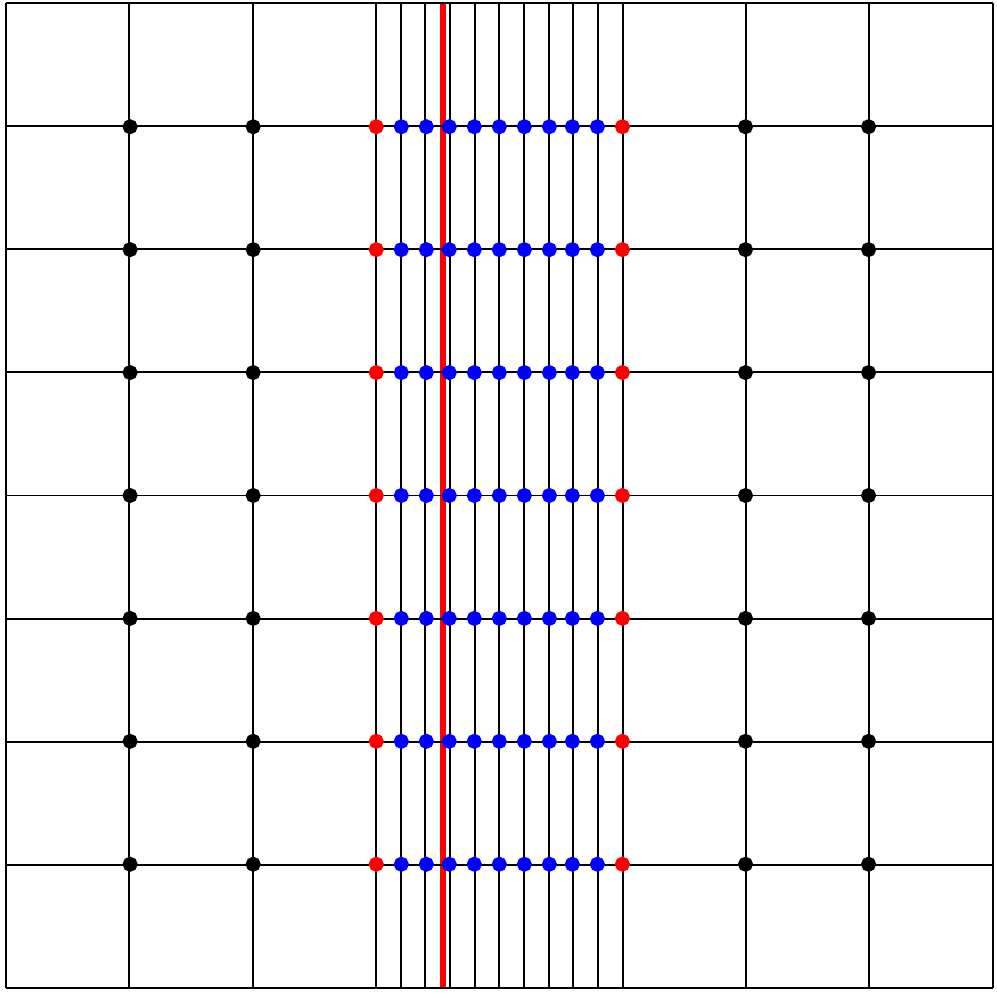}
\caption{A two-grid with a line interface (red solid).
There are three types of grid points: coarse grid points (black dotted), fine grid points in the $x$-direction (blue dotted), and  boarder grid points (red dotted) that  connec the coarse and fine meshes in the $x$-direction.
The refinement ratio is $r=5$, and the refinement tube width is $\lambda=1.5$ in the figure.}
\label{fig:2dmesh}
\end{figure}

In the strip $|x_i - \alf | \le  \lambda h$ as shown in Figure \ref{fig:2dmesh}, that is, a fine grid in the $x$ direction but the same coarse grid size in the $y$ direction,
 with the immersed interface method (IIM) as discussed  in Subsection~\ref{sec-IIM},   the finite difference equation  that is second order in $x$ and fourth order in $y$  in the tube can be written as
\eqm
  \frac{U_{i-1,j}- 2 U_{ij} + U_{i+1,j}}{h_f^2} + \lp 1 + \fr{h^2}{12} \delta_{yy}^2 \rp^{-1}
\delta_{yy}^2 U_{ij} = f_{ij} + C_{ij},
\enm
where $\delta_{yy}^2 $ is the standard three-point central finite difference operator in the $y$-direction and $C_{ij}$ is the correction term which is zero almost everywhere except at the  two grid points ($x_j$ and $x_{j+1}$) neighboring the interface~$\alf$.  Note also that $C_{ij}$  is independent of $y$.  Thus, the discretization can further be written as
\eqm
 \lp 1 + \fr{h^2}{12} \delta_{yy}^2 \rp  \frac{U_{i-1,j}- 2 U_{ij} + U_{i+1,j}}{h_f^2} +
\delta_{yy}^2 U_{ij} = \lp 1 + \fr{h^2}{12} \delta_{yy}^2 \rp   \lp \fr{\null}{\null} f_{ij} + C_{ij} \rp,
\enm
which leads to
\eqmno
 &&  \!\!\!\!\! \frac{U_{i-1,j}- 2 U_{ij} + U_{i+1,j}}{h_f^2}  +  C_{ij} + \frac{h_y^2}{12 h_f^2} \left ( \frac{\null}{\null} U_{i-1,j-1} - 2  U_{i-1,j} + U_{i-1,j+ 1} -2 U_{i,j-1} \right. \\ \eqsp
 && \quad   \left.   \frac{\null}{\null} + 4  U_{i,j}  - 2  U_{i,j+ 1}  +    U_{i+1,j-1} - 2  U_{i+1,j} + U_{i+1,j+ 1}   \right )  =  \frac{1}{12} \left ( f_{i,j-1} +10  f_{i,j} + f_{i,j+ 1}  \frac{\null}{\null}  \right ).
\enmno
The  finite difference equation involves the nine-point compact stencil, and   is second order accurate  in $x$ and fourth order accurate in $y$, respectively.

For the boarder grid points as shown in Figure \ref{fig:2dmesh}, following the similar derivation process for 1D boarder grid points in subsection \ref{subsec},
we can obtain a compact finite difference  equation at a boarder grid point $(x_i, y_j)$. The finite difference coefficients for $U_{i,j}$ and the
linear combination of $f_{i,j}$ are listed below,
        {$$
        U_{i,j}:\quad
        \frac{1}{6 \sigma h_y^2}
        \begin{bmatrix}
          h_2 (h_1 h_2 + h_y^2 + \mu)       & (h_1 + h_2) (3 h_1 h_2 - h_y^2 + \nu) & h_1 (h_1 h_2 + h_y^2 - \mu) \\
          -2 h_2 (h_1 h_2 - 5 h_y^2 + \mu)  & -2 (h_1 + h_2) (3 h_1 h_2  + 5 h_y^2 + \nu) & -2 h_1 (h_1 h_2 - 5 h_y^2 - \mu) \\
          h_2 (h_1 h_2  + h_y^2 + \mu)      & (h_1 + h_2) (3 h_1 h_2 - h_y^2 + \nu ) & h_1 (h_1 h_2 + h_y^2 - \mu)
        \end{bmatrix},
        $$}
        $$
        f_{i,j}:\quad \frac{1}{12 \sigma}
        \begin{bmatrix}
             0          &   \sigma        & 0    \\
             2 h_2 (h_1 h_2 - \nu) & 2 (h_1 + h_2)^3  & 2 h_1 (h_1 h_2 + \nu)   \\
             0          &   \sigma        & 0
        \end{bmatrix},
        $$
where $h_{1} = x_i - x_{i-1}, h_{2} = x_{i+1} - x_i, h_y = y_{j} - y_{j-1} = y_{j+1} - y_j$ and $\sigma = h_1 h_2 (h_1 + h_2),\ \mu = h_1^2 - h_2^2,\ \nu = h_1^2 + h_2^2.$

Now we have defined finite difference equations at all grid points with local truncation errors being $O(h^4+h_f^2)$. The discretize system of equations satisfy the discrete maximum principle, and thus, the global accuracy is also of $O(h^4+h_f^2)$.

\begin{example}
  This example is adapted from the 1D example in Subsection \ref{sub1dex}. A Poisson equation with a line interface $x=\alf=\frac{33}{70}$ on a domain $0<x<1$, $0<y<1$ with the true solution,
\eqm
 u(x,y) = \left\{ \begin{array}{ll}
x (\alf-1) + \sin (\pi y) & \textrm{if $x \leq \alf$}\\ \eqsp
\alf (x-1) +  \sin (\pi y) & \textrm{if $x>\alf$.}
\end{array} \right.
\enm
{\em The source term is $f(x,y) = -\pi^2  \sin (\pi y)$. Across the interface, the jump conditions are $[u]_{x=\alf}=0$ and $\left[\frac{\partial u}{\partial x} \right]_{x=\alf}= 1$. }
\end{example}

\begin{figure}[htbp]
\begin{minipage}[c]{3in}
     \includegraphics[width=1.1\textwidth]{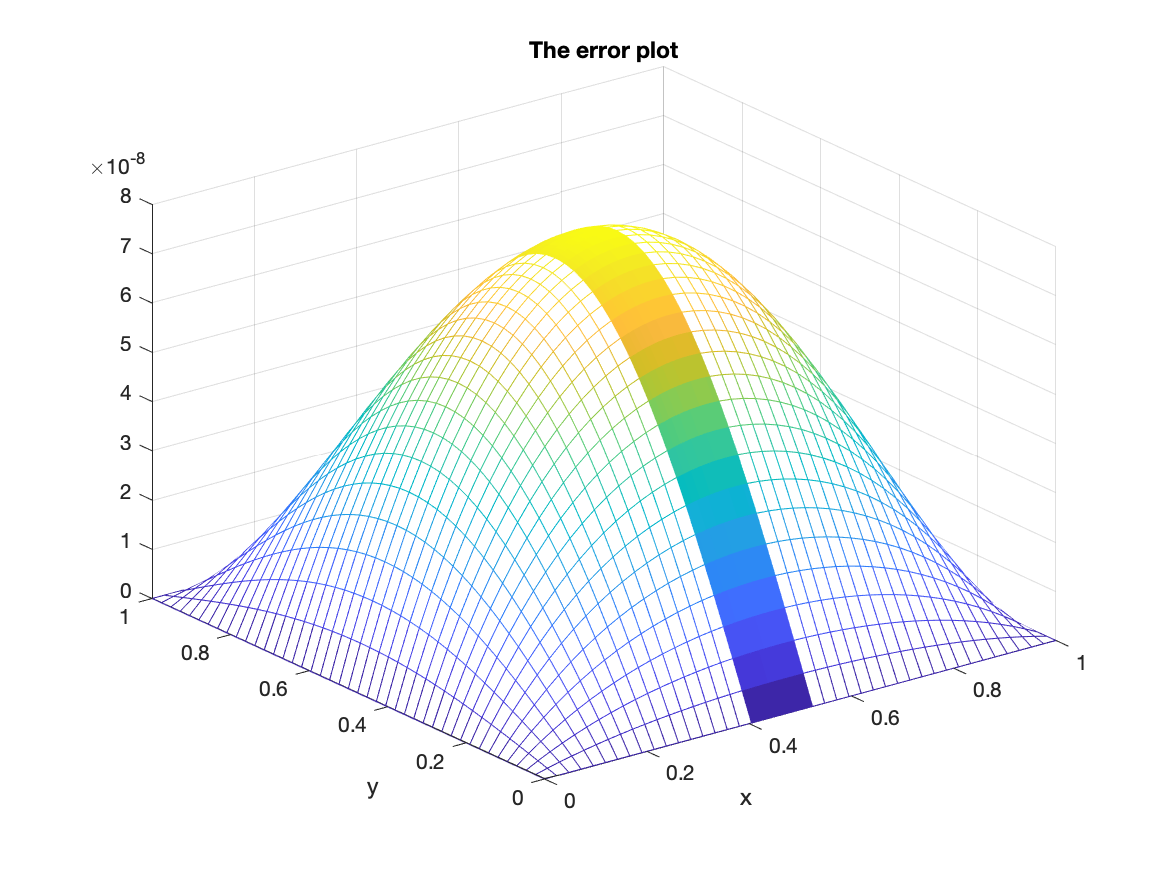}

\vspace{-0.45cm}
{\small A two grid along the $x$-axis and the error plot with $N=42$, $ \lambda =2$, and $h_f =h^2$.  The maximum error is $7.8476\times 10^{-8}$.}
\vthin

\end{minipage}
$\null \qquad$
\begin{minipage}[t]{2.3in}
$\null$

  \vspace{-1.65cm}
$\null \quad$
\begin{tabular}{|c|| c|c|}
\hline
     $N$ &   $\|E \|_{\infty}$(two-grid) & order \\  \hline
     $6$ & $  1.6355 \times 10^{-4} $ &  \\  \hline
     $12$ & $ 1.1807 \times 10^{-5} $ &3.9585   \\  \hline
     $24$ & $ 7.3631 \times 10^{-7} $ & 4.0032 \\  \hline
     $42$ & $ 7.8476  \times 10^{-8} $ & 4.1156\\  \hline
   \end{tabular}

 \vthin

 \vthin

 A grid refinement analysis for the 2D example with a line interface.  Fourth order convergence is confirmed  when $h_f= h^2$.

  \end{minipage}
  \caption{An error plot of the computed solution of the 2D example with a line interface in which two grids in the $x$ direction are visible.  } \label{tab:2D_1D}
 \end{figure}

In the left plot of Figure~\ref{tab:2D_1D},  we show  an error plot with $N=42$ and $h_f = h^2$  in which  two grids with different mesh sizes  are also visible.
 In the right table, we show a grid refinement analysis which indicates fourth order convergence in terms of $h$ when $h_f = h^2$ as expected.

\section{A two-grid  method for general 2D interface problems} \label{sec:2D}

For general two-dimensional elliptic interface problems, it is much more challenging to develop  two-grid methods  with general curved interfaces. We  consider a general two-dimensional elliptic interface problem of the  following form,
\eqm
\dsp \nabla \cdot \left(\kappa \nabla u \right ) + K u =  f(\mathbf{x}), \quad   \mathbf{x}\in R,
\enm
together with jump conditions across interface $\Gamma$
\eqm \label{elliptic-jump}
\dsp \left [ u \right]_{\Gamma}   = w(s), \qquad \left [ \kappa  \frac{\partial u}{\partial n} \right]_{\Gamma} = v(s),
\enm
and a boundary condition along $\partial R$, where $R$ is a rectangular domain,
$\Gamma\in C^2$ is a closed smooth interface that can be parameterized by a one-dimensional variable $s$, say the arc-length. Within the domain $R$; $w(s)\in C^2$ and $v(s)\in C^1$ are two functions defined along the interface $\Gamma$. Note that, when $w(s)=0$, the problem can be written as
 \begin{eqnarray}
  && \dsp \nabla \cdot \left(\kappa \nabla u \right )  + K  u =  f(\mathbf{x}) + \int_{\Gamma} v(\mathbf{X}(s) )\delta\left ( \mathbf{x}-\mathbf{X}(s) \right ) ds , \quad   \mathbf{x}\in R.
\end{eqnarray}
Second order methods  on Cartesian grids have been developed in the literature. The challenge here is to develop higher (third or above) order compact methods.
It is more useful but challenging for the case when $\kappa$ is discontinuous.
In this paper, we propose an alternative approach to a uniform mesh, that is, a two-grid method to obtain high order accurate methods, which is not only simpler but also avoids higher order derivatives required for jump conditions along the interface as well as the source terms and the solution.

We assume that the interface problem is defined in a rectangular domain
$\Omega = [a,\,b]\times [c,\,d]$. We start with a coarse Cartesian grid,
$x_i = a + ih$, $y_j=c+jh$, $i=0,1,\cdots,M$, $j=0,1,\cdots, N$.
The interface $\Gamma$ is implicitly represented by the zero level set of a Lipschitz
continuous function $\varphi(x,y)$:
\eqm
\Gamma = \left \{\,  ( x, y ), \fr{\null}{\null}\varphi(x,y)=0 \right \}.
\enm
 In the discrete case, $\varphi(x,y)$ is defined at grid points as $\varphi(x_i,y_j)$.
Often $\varphi(x,y)$ is the signed distance function from~$\Gamma$, or its approximations.

\begin{figure}[hpbt]
\centering
\includegraphics[width=0.4\textwidth]{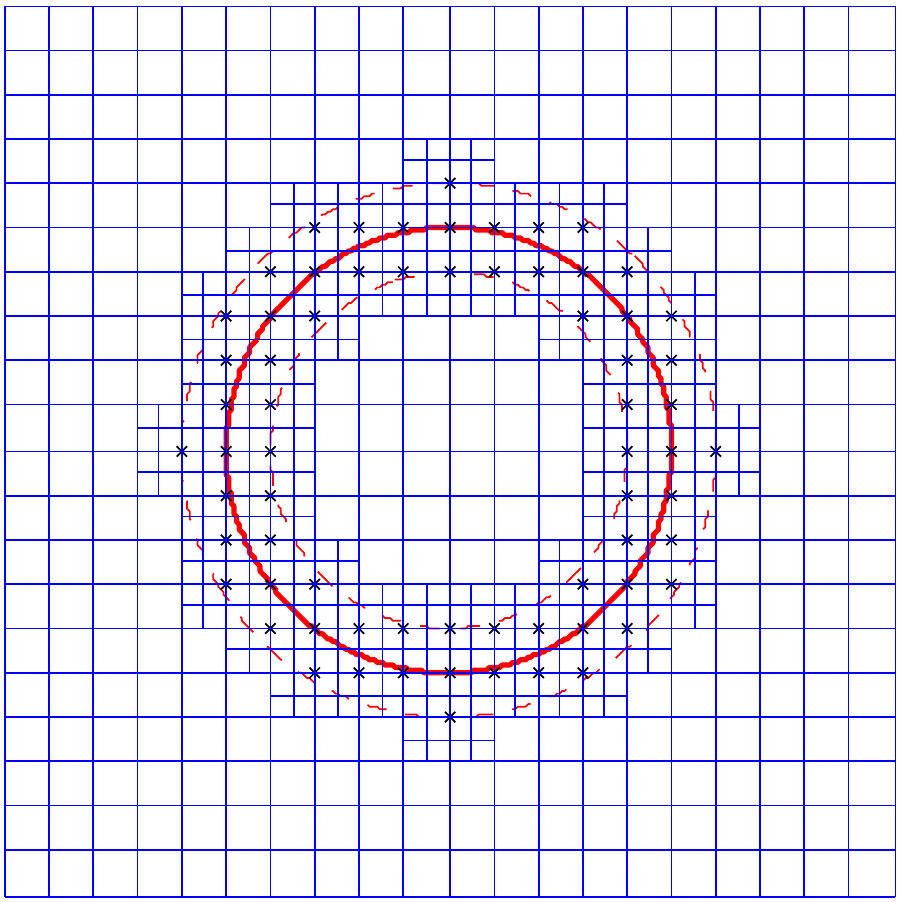}
\caption{A two-grid with a circular interface (red solid).
Parent grid points (starred) are selected within the tube
$|\varphi_{ij}|=|\varphi(x_i,y_j)|\le h$ (red dashed).
The refinement ratio is $r=2$, and the refinement tube width is $\lambda=1$.}
\label{fig:amr_mesh}
\end{figure}

To generate a finer mesh around the interface $\Gamma$, we first select
parent grid points within a tube of the interface according to \eqm
 |\, \varphi(x,y)\, | \leq \lambda h,
\enm where $\lambda$ is a control coefficient to adjust the
width of the refinement tube. When a grid point $\textbf{x}_{ij}=(x_i,y_j)$ within the tube is selected, we build a refined mesh with  a finer  mesh $h/r$ ($r$ is refinement ratio, $r=2$ or $4$, $\cdots$) within the square: $|x-x_i|\le h$ and $|y-y_j| \le h$.
Generating a refined square for every such grid points yields a refined
region around the interface.  
Figure~\ref{fig:amr_mesh} shows an example of the refinement mesh around a
circular interface.
The diagram at the right in Figure~\ref{fig:hanging_diag} is an illustration of some grid points in the coarse and finer grid with $r=4$.

Once we have the two grids, we use the standard fourth-order compact finite difference scheme at interior coarse grid points; and the standard second order method at interior regular fine grid points; and second order maximum principle immersed interface method (IIM) \cite{li-ito,li:book} at the fine irregular grid points within the tube. In all the three cases, the finite difference coefficients satisfy the sign property that guarantee the discrete maximum principle.

\subsection{A new super-third order FD discretization for hanging nodes}

One of the challenges for two-grid or adaptive mesh methods is how to treat  hanging nodes, which has been intensively discussed in the literature. There are two main concerns about  the  discretization: accuracy, and the stability that is often more difficult to study. One sufficient (stronger) condition is to check the sign property of  finite difference coefficients.
Our new idea is to use  the values of the solution,  the source term $f$, and the partial differential equations to derive  finite difference  equations that can be accurate for highest polynomials possible while satisfying the sign property.
We also want the finite difference scheme to be robust so that the coefficients just need to be computed once for PDEs with constant coefficients.
\begin{figure}[htbp]
\centering
     \includegraphics[width=0.4\textwidth]{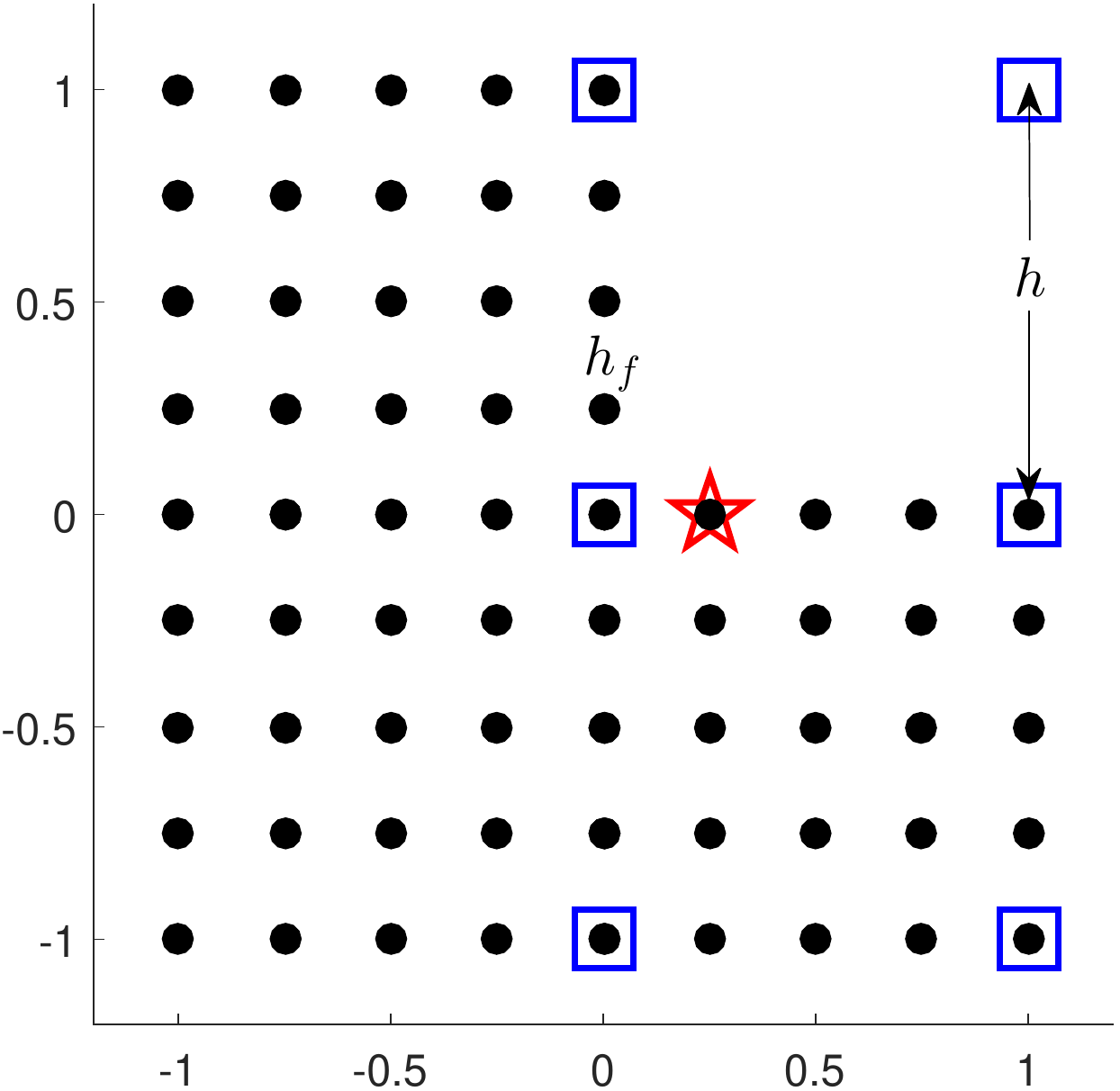}
     \caption{A schematic diagram of 7-point finite difference stencil for  a typical hanging node (marked as a red five-star) when $r=4$ ($h_f=h/4$).}\label{fig:hanging_diag}
\end{figure}

There are various approaches in designing finite difference equations at hanging nodes in the literature. One commonly used strategy is to use ghost grid points\ignore{as marked as '0's  in  the right diagram in Figure~\ref{fig:hanging_diag}},   and then apply interpolation schemes. We refer the readers to \cite{li_song_AMR}  and references therein. Often, the developed finite difference equations  are consistent. However, it is difficult to show the stability,  and the convergence of the finite difference method.

In this paper, we employ a new strategy, see also \cite{li-pan-hoc21},  and some new ideas to develop finite difference equations at  hanging nodes for elliptic PDEs
\eqm
 \kappa \Delta  u  + K u =  f(\mathbf{x}),
\enm
assuming that $\kappa$ and $K$ are constants. We use the left diagram in Figure~\ref{fig:hanging_diag} to illustrate the method in which $r=4$. We develop a finite different equation for the hanging node  $(x_{i_c},y_{j_c})$ (marked as a red five-star). In the figure, solid black circles are fine grid points,  small rectangles are coarse ones. A grid point can be both fine and coarse.  In fact, for the PDE above with constant coefficients, the coefficients of the finite difference equation and the weights for the source term are invariant in terms of  the relative locations and the PDE. So we just need to use a reference geometry in the derivation as illustrated in Figure~\ref{fig:hanging_diag}, in which
there are three such hanging nodes along the line $x=0$, and another three along the line $y=0$.  %

After some research, testing,  and reasoning based on symmetry arguments, we come to a conclusion that a seven-point stencil as demonstrated in  Figure~\ref{fig:hanging_diag}  is an ideal choice as we can see later in this section.
In order to derive  a high order compact (HOC) finite difference equation (7-point stencil) at the hanging node (marked as the red star), we use its neighboring six grid points (marked in blue square). The finite difference stencil  has one coarse grid point, and six coarse and fine ones.
 Denote the coordinates at hanging node $(x_{i_c},y_{j_c})$ as $(h_2,0)$.
Then, the grid points involved are $(0, \pm h)$, $(h,\pm h)$, $(0,0)$, $(h, 0)$ and $(h_2,0)$.
Note that $h_2 = j h_f$ for $j=1,2,\cdots, r-1$.
It is obvious that the finite difference  equations for hanging points with $j=k$ and $j=r-k$ are  the same using the symmetry arguments, which has been verified by symbolic computations. We design the finite difference equation as:
\eqml{4th-hoc}
&& \!\!\!\!  \!\!\!\!  \dsp  \sum_{i_k=0}^{1} \sum_{j_k=-1}^{1} \!\!  \alpha_{i_k,j_k} U_{i+i_k,j+j_k}   +  \alpha_{i_c,j_c} U_{i_c,j_c} =  \sum_{i_k=0}^{1} \sum_{j_k=-1}^{1}  \!\beta_{i_k,j_k} \, f(x_{i+i_k},y_{j+j_k}) +
\beta_{i_c,j_c} \, f(x_{i_c},y_{j_c}) , \\ \eqsp
  && \dsp  \sum_{i_k=-1}^1 \sum_{j_k=-1}^1 \!\!\! \beta_{i_k,j_k}  +  \beta_{i_c,j_c} = 1.
\enml
Note  that the {\em indexes} \ are not real ones for actual algorithms but just for the derivation of the finite difference equation. The coefficients are just needed to be computed once for all for constant $\kappa$ and $K$.

Define the `local truncation error' as
\eqml{truc}
 T_{ {i_c},{j_c}} & = & \dsp  \sum_{i_k=0}^{1} \sum_{j_k=-1}^{1} \!\!  \alpha_{i_k,j_k} u(x_{i+i_k},y_{j+j_k} )  +  \alpha_{i_c,j_c}  u(x_{i_c},y_{j_c} )  \\ 
&& \null  \quad \qquad \dsp - \sum_{i_k=0}^{1} \sum_{j_k=-1}^{1}  \!\beta_{i_k,j_k} \, f(x_{i+i_k},y_{j+j_k}) -\beta_{i_c,j_c} \, f(x_{i_c},y_{j_c}) ,
\enml
where $u(x,y)$ is the true solution to the boundary value problem.
To determine the finite difference coefficients $\alpha$'s and the weight coefficients $\beta$'s, we expand all $u$ and $f$ terms at $(x_{i_c},y_{j_c})$ up to all fourth order derivatives of $u(x,y)$ and second order derivatives of $f(x,y)$; then we represent $f(x,y)$ terms in terms of $u(x_c,y_c)$ and its partial derivatives at the $(x_c,y_c)$ using the following relations,
\begin{equation}
 \frac{\partial^{k_1+k_2}  f}{\partial x^{k_1} \partial y^{k_2} }  = \kappa\left(  \frac{\partial^{k_1+k_2+2} u }{\partial x^{k_1+2} \partial y^{k_2} }  +  \frac{\partial^{k_1+k_2+2}  u }{\partial x^{k_1} \partial y^{k_2+2} } \right )  + K  \frac{\partial^{k_1+k_2} u }{\partial x^{k_1} \partial y^{k_2} } , \quad 0\le  k_1+k_2\leq 2.
\end{equation}
With the above relations, we can rewrite the local truncation error as
\eqm
  T_{ {i_c},{j_c}} = \sum_{ 0 \leq k_1+k_2 \leq 4}  \!\!\!\!\! L_{k_1,k_2} \, \left.   \frac{\partial^{k_1+k_2}  u}{\partial x^{k_1} \partial y^{k_2} } \right |_{(x_c,y_c)}
  + O\left(\| {\bfalf} \|_{\infty} h^5 + \| {\bfbeta} \|_{\infty} h^3 \right ),
\enm
where $\| {\bfalf} \|_{\infty}$ is the maximum norm of all $\alpha$'s coefficients and so on.
 We want  $T_{ {i_c},{j_c}}$ to be as small as possible in magnitude in terms of $h$.
 So we set the coefficients of $\frac{\partial^{k_1+k_2}  u}{\partial x^{k_1} \partial y^{k_2} }$
 to be zero for all the  $h^k$ coefficients,  which leads to a linear system of equations for the coefficients $\alpha$'s and the weight coefficients $\beta$'s.
 The equality constraint in \eqn{4th-hoc} prevents the trivial solution (all zero coefficients).
\ignore{ Although we have the same number of  unknowns (16)  as the number of  equations} If we make all the coefficients of  $\frac{\partial^{k_1+k_2}  u}{\partial x^{k_1} \partial y^{k_2} }$  to be zero for all $0\le k_1+k_2\le 4$,
we would obtain 16 equations with 14 unknowns (7 $\alpha$'s and 7 $\beta$'s). It turns out that the system of equations is not consistent. Thus, we give up two equations corresponding to $x^4$ and $y^4$ terms to have $14$ equations and $14$ unknowns. The system of equations then is consistent and have infinity number of solutions. We can impose some additional conditions such as the symmetry, sign property to get  desired  solutions.

\begin{table}[ht]
\caption{Finite difference coefficients $h^2 \alf_{i_k, j_k}$ at hanging point $(j h_f, 0)(j=1,2,\cdots, r-1)$ for $r=2$, $4$, $8$.
The finite difference coefficients for $j=k$ are the same as those for $j=r-k$ except that they are in the reverse order, 
and the coefficients are the same for any $j$ and $r$'s when $j/r = \textrm{constant}$.} \label{hanning_coef_u}
\begin{center}
\begin{tabular}{|c|c|c c c c c c |c |}
    \hline
    $r$  & $j$ & \multicolumn{6}{c|}{$h^2 \alf_{i_k,j_k}$} & $h^2 \alf_{i_c, j_c}$ \\ \hline
    2 &1&  { \color{blue} 1/2 }    &   { \color{blue} 1/2 }   &       { \color{blue}  3  }   &         { \color{blue} 3   }   &       { \color{blue}  1/2  }   &      { \color{blue}  1/2  }     &    { \color{blue} -8 }   \\  \hline
     \multirow{3}{*}{4}
   & 1& \textcolor[rgb]{1,0,0}{7/12}    &  \textcolor[rgb]{0.9,0,0}{   5/12  }    & \textcolor[rgb]{0.8,0,0}{   41/6 }    &    \textcolor[rgb]{0.7,0,0}{   11/6  }    &   \textcolor[rgb]{0.6,0,0}{   7/12  }   &  \textcolor[rgb]{0.5,0,0}{   5/12 }   &  { \color{red}   -32/3 }  \\
   & 2&  { \color{blue} 1/2 }    &   { \color{blue} 1/2 }   &       { \color{blue}  3  }   &         { \color{blue} 3   }   &       { \color{blue}  1/2  }   &      { \color{blue}  1/2  }     &    { \color{blue} -8 }   \\
   & 3& \textcolor[rgb]{0.5,0,0}{   5/12  }    &  \textcolor[rgb]{0.6,0,0}{   7/12}     &    \textcolor[rgb]{0.7,0,0}{  11/6}     &      \textcolor[rgb]{0.8,0,0}{ 41/6}      &    \textcolor[rgb]{0.9,0,0}{   5/12 }     &    \textcolor[rgb]{1,0,0}{  7/12  }   &      { \color{red}   -32/3 } \\  \hline
    \multirow{7}{*}{8}
    & 1& 5/8       &     3/8        &   59/4      &     43/28      &     5/8       &     3/8    &      -128/7  \\
    & 2& \textcolor[rgb]{1,0,0}{7/12}    &  \textcolor[rgb]{0.9,0,0}{   5/12  }    & \textcolor[rgb]{0.8,0,0}{   41/6 }    &    \textcolor[rgb]{0.7,0,0}{   11/6  }    &   \textcolor[rgb]{0.6,0,0}{   7/12  }   &  \textcolor[rgb]{0.5,0,0}{   5/12 }   &  { \color{red}   -32/3 }  \\
    & 3& \textcolor[rgb]{0,1,0}{13/24}    &      \textcolor[rgb]{0,0.9,0}{11/24}     &     \textcolor[rgb]{0,0.8,0}{17/4 }   &      \textcolor[rgb]{0,0.7,0}{137/60  }  &      \textcolor[rgb]{0,0.6,0}{13/24 }    &     \textcolor[rgb]{0,0.5,0}{11/24 }   &   \textcolor[rgb]{0,1,0}{ -128/15 }\\
    & 4&  { \color{blue} 1/2 }    &   { \color{blue} 1/2 }   &       { \color{blue}  3  }   &         { \color{blue} 3   }   &       { \color{blue}  1/2  }   &      { \color{blue}  1/2  }     &    { \color{blue} -8 }   \\
    & 5& \textcolor[rgb]{0,0.5,0}{11/24  }    &   \textcolor[rgb]{0,0.6,0}{ 13/24 }   &   \textcolor[rgb]{0,0.7,0}{  137/60  }    &   \textcolor[rgb]{0,0.8,0}{ 17/4}      &     \textcolor[rgb]{0,0.9,0}{11/24  }   &    \textcolor[rgb]{0,1.0,0}{ 13/24}    &  \textcolor[rgb]{0,1,0}{   -128/15} \\
    & 6& \textcolor[rgb]{0.5,0,0}{   5/12  }    &  \textcolor[rgb]{0.6,0,0}{   7/12}     &    \textcolor[rgb]{0.7,0,0}{  11/6}     &      \textcolor[rgb]{0.8,0,0}{ 41/6}      &    \textcolor[rgb]{0.9,0,0}{   5/12 }     &    \textcolor[rgb]{1,0,0}{  7/12  }   &      { \color{red}   -32/3 } \\
    & 7&       3/8      &      5/8      &     43/28   &       59/4    &        3/8      &      5/8      &    -128/7 \\  \hline
\end{tabular}
\end{center}
\end{table}

\begin{table}[ht]
\caption{Weight coefficients $\beta_{i_k,j_k}$ at hanging point $(j h_f, 0)$ for different $j/r(r = 2,4,8)$.  } \label{hanning_coef_f}
\begin{center}
\begin{tabular}{|c|c c c c c c |c |}
    \hline
     $j/r$ & \multicolumn{6}{c|}{$\beta_{i_k,j_k}$} & $\beta_{i_c, j_c}$ \\ \hline
     1/2& 0& 0&   \textcolor{blue}{1/2}&   \textcolor{blue}{1/2}& 0& 0&0 \\
     1/4& 0& 0&  \textcolor{red}{7/12}&  \textcolor{red}{5/12}& 0& 0&0 \\
     1/8& 0& 0&   5/8&  3/8& 0& 0&0 \\
     3/8& 0& 0& \textcolor{green}{13/24}& \textcolor{green}{11/24}& 0& 0&0 \\
\hline
\end{tabular}
\end{center}
\end{table}

With the help of the Matlab symbolic package, we have analytically found one `good' set of  finite difference coefficients and  weights of the source term when $\kappa=1$ and $K=0$ as listed in Table~\ref{hanning_coef_u} and Table~\ref{hanning_coef_f} for $r=2,4,8$.
The last column in Table~\ref{hanning_coef_u}  lists the coefficients $\alpha_{i_c,j_c}$ at the diagonals. 
The coefficients for $r=16$ are listed in Table \ref{r12}.
From these talbles, we have verified  following properties of the finite difference scheme, which we referred as a `good' set of the coefficients.
\bi
 \item The finite difference coefficients satisfy the sign property, that is, all diagonals are negative while off-diagonals are positive. So the coefficient matrix for the finite difference equations is an M-matrix.
 \item  The finite difference coefficients and weights at hanging point $(j h_f, 0)$ are the same for any $j$ and $r$ with $j/r = \textrm{constant}$.
\item The finite difference coefficients and weights at hanging point $(j h_f, 0)$ for $j=k$ are the same as those for $j=r-k$ except they are in the reverse order.
 \item Only two weight coefficients are nonzero: $\beta_{0,1}$ and $\beta_{1,1}$, which are positive and their sum is one.
  \item The discretization at a hanging node is super-third order accurate. In our numerical tests,  the errors  from the discretization at hanging nodes are  less dominant compared with those at irregular grid points near the interfaces in the finer mesh.
\ei

\begin{table}[ht]
\caption{Finite difference coefficients $h^2 \alf_{i_k, j_k}$ and weight coefficients  $\beta_{i_k, j_k}$ at the hanging grid point $(j h_f, 0)(j=1,2,\cdots, r/2)$  when  $r=16$. The coefficients for $j=9,10, \cdots, 15$, which are not listed here, are the same as those for $j=7,6,\cdots, 1$, respectively.
Similar to the cases for $r=2, 4, 8$ (see Table \ref{hanning_coef_f}), only two weight coefficients $\beta_{0,0}$ and $\beta_{1,0}$ are nonzero. }\label{r12}
\begin{center}
\begin{tabular}{|c|c c c c c c |c | cc |}
    \hline
     $j$ & \multicolumn{6}{c|}{$h^2 \alf_{i_k,j_k}$} & $h^2 \alf_{i_c, j_c}$  & $\beta_{0,0}$ & $\beta_{1,0}$ \\ \hline
     1&  31/48& 17/48& 737/24&   57/40& 31/48& 17/48& -512/15& 31/48& 17/48 \\
     2&    5/8&   3/8&   59/4&   43/28&   5/8&   3/8&  -128/7&   5/8&   3/8 \\
     3&  29/48& 19/48& 227/24& 521/312& 29/48& 19/48& -512/39& 29/48& 19/48 \\
     4&   7/12&  5/12&   41/6&    11/6&  7/12&  5/12&   -32/3&  7/12&  5/12 \\
     5&   9/16&  7/16& 211/40&  179/88&  9/16&  7/16& -512/55&  9/16&  7/16 \\
     6&  13/24& 11/24&   17/4&  137/60& 13/24& 11/24& -128/15& 13/24& 11/24 \\
     7&  25/48& 23/48&593/168&  187/72& 25/48& 23/48& -512/63& 25/48& 23/48 \\
     8&    1/2&   1/2&      3&       3&   1/2&   1/2&      -8&   1/2&   1/2 \\   \hline
\end{tabular}
\end{center}
\end{table}

\subsection{Convergence analysis}

Now we discuss the convergence of the two-grid method. The proof is based on the discrete maximum principle characterized by the M-matrix property and the local truncation errors.

\begin{theorem} \label{thm-A}
Let $U_{ij}$ be the finite difference solution obtained from the two-grid method for $\kappa \Delta u = f$ with a linear boundary condition (Dirichlet, Neumann, or Robin), and at least one point  Dirichlet  boundary condition.
 Assume that the finite difference coefficients satisfy the sign property at all the grid points with $\| {\bfalf} \|_{\infty} \le C/h^2$ and $ \| {\bfbeta}\|\le C$,
  and the solution to the elliptic interface problem is {\em piecewise smooth} with the needed regularity, then,
\eqm
  \left \| u(x_i,y_j)- U_{ij} \right \|_{\infty}  \le C \left ( O(h^3) + O\left ({h/r}\right )^2  \right ).
\enm
Here, $C$ is a generic constant, $h$ is the step size of the coarse mesh and $h/r$ is that of the fine mesh.
\end{theorem}

\begin{proof}
  Note that, in the interior domain, the regularity of the solution depends on the source term $f$.  In the neighborhood of coarse grid points, the needed regularity is $u(x,y)\in C^6$ and we have $| T_{ij} | \le C h^4$.   In the neighborhood of  hanging nodes, the needed regularity is $u(x,y)\in C^5$ and we  know that the local truncation errors are bounded by $| T_{ij}| \le C h^3$.   In the neighborhood of  fine grid points including irregular grid points near the interface, the needed regularity is $u(x,y)$ is piecewise $C^3$ and we have $| T_{ij} | \le C h/r$.
 Note also that the coefficient matrix of the finite difference equations is an M-matrix. Thus, from  Theorem~6.1 and Theorem~6.2 of Morton~\&~Mayer's book, \cite{morton-mayers}  we obtain the conclusion.
\end{proof}

\begin{remark} $\null \!\!\!\!$
\bi
    \item We have obtained  explicit expressions for the coefficients and the weights for $r=2,4, \cdots, 16$ that satisfy the conditions of the theorem. Thus, the convergence is guaranteed and confirmed for those $r$'s.
    \item In our numerous numerical tests, the largest errors  seem to occurr near the interface rather than at the hanging nodes. We conjecture that the error should be bounded by $ \left \| u(x_i,y_j)- U_{ij} \right \|_{\infty}  \le C \left ( O(h^4) + O\left ({h/r}\right )^2  \right )$ instead of $O(h^3)$.
\ei
\end{remark}

\subsection{Numerical experiments for general  2D interface problems}

We have tested the two-grid method for various benchmark problems in the literature. We list some of the results in this subsection.

\begin{example} \label{ex1}  Peskin's IB model in two dimensions.
\end{example}
This example was first presented in \cite{rjl-li:sinum}, and later was employed
 by the AMR-IB method in \cite{roma:thesis,Roma-Peskin-Berger}.
The  partial differential equation is
\eqm
 u_{xx}+u_{yy}=\int_{\Gamma}{2\, \delta(x-X(s))\delta(y-Y(s))\, ds},
\enm
where $\Gamma$ is a circle  interface  with the radius $r=0.5$  within the square domain $R=(-1,1)\times(-1,1)$.  The problem can be written  equivalently as,
\eqm
 u_{xx}+u_{yy}=0, \quad \mbox{on} \quad \Omega\backslash\Gamma,
\qquad [u]_{\Gamma}=0, \quad \left [\frac{\partial u}{\partial n}
\right]_{\Gamma}=2,
\enm
in the domain $R$.  The true solution to the interface problem is
\eqm
 u(x,y) = \left\{ \begin{array}{ll}
1 & \textrm{if $r\leq 1/2$}\\ \eqsp
1+\log(2r) & \textrm{if $r>1/2,$}
\end{array} \right.
\enm
where $r=\sqrt{x^2+y^2}$. The Dirichlet boundary condition is  applied according
to the true solution along the boundary of the square.

\begin{table}[hptb]
\centering \caption{Comparison of numerical results using IB, IIM, and the two-grid method for Example~\ref{ex1}, with $r=2,4,8$.} \label{tab1}

\vthin
\vthin

\begin{tabular}{|c|c|c||c c|}
\hline
     $N$ & IB & IIM & $\|E \|_{\infty}$ (coarse ) & $\|E \|_{\infty}$ (fine) \\  \hline
      20  & $3.614\times 10^{-1}$  & $2.3908 \times 10^{-3}$ &  $1.6088\times 10^{-4}$    & $ 3.1425\times 10^{-4}$   \\ \hline
     \multirow{3}{*}{40} &\multirow{3}{*}{ $2.6467\times 10^{-2}$} &\multirow{3}{*}{$8.3461 \times 10^{-3}$} &  $3.8203 \times 10^{-5}$  &
     $7.8461   \times 10^{-5}$ \\ \cline{4-5}
     & & & $5.9363  \times 10^{-6}$ &  $1.6361\times 10^{-5}$     \\ \cline{4-5}
     & & &  $1.9566 \times 10^{-6}$   &   $4.6673 \times 10^{-6}$    \\ \hline
 \multirow{3}{*}{80} &\multirow{3}{*}{ $1.3204\times 10^{-2}$} &\multirow{3}{*}{$2.4451 \times 10^{-4}$} &   $6.7665 \times 10^{-6}$    &   $1.5000 \times 10^{-6}$  \\ \cline{4-5}
     & & & $2.0704  \times 10^{-6}$  &   $3.9415 \times 10^{-6}$    \\ \cline{4-5}
     & & & $3.6112\times 10^{-7}$  &$  8.1338\times 10^{-7}$   \\   \hline
    \multirow{3}{*}{160} &\multirow{3}{*}{ $6.6847\times 10^{-3}$} &\multirow{3}{*}{$6.6573 \times 10^{-4}$} &  $2.2361\times 10^{-6}$    & $3.7708\times 10^{-6}$ \\ \cline{4-5}
     & & & $3.8562\times 10^{-7}$  &$ 7.5454\times 10^{-7}$    \\ \cline{4-5}
     & & & $2.0934\times 10^{-7}$  &$  2.5683\times 10^{-7}$     \\  \hline
     \multirow{3}{*}{320} &\multirow{3}{*}{ $3.3393\times 10^{-3}$} &\multirow{3}{*}{$1.5672 \times 10^{-5}$} & $7.9409\times 10^{-7}$ &$ 9.3940\times 10^{-7}$   \\ \cline{4-5}
     & & & $1.8957\times 10^{-7}$ &  $  2.2149\times 10^{-7}$ \\ \cline{4-5}
     & & & 5$.0883\times 10^{-8}$ &$   5.9343\times 10^{-8}$ \\ \hline
\hline

\end{tabular}
\end{table}

In Table~\ref{tab1}, we show some numerical results. We can see that the the numerical results of the two-grid method performs much better compared with the immersed boundary (IB) method \cite{peskin-review}  and the IIM even with small $r=2,4,8$. The errors in the coarse grid  using the fourth-order method are comparable with that at the fine grid using the second order IIM. From the result using a uniform mesh, we observe fourth order convergence  using the two-grid method  when  $r=4$.

\begin{example} \label{ex-flower}
  An example of discontinuous coefficient and solution, and a complicated interface.
\end{example}
This example is from \cite{li:fast, li:book} where the augmented IIM was developed for elliptic interface problems with piecewise constant coefficients. We consider the elliptic interface problem $(\kappa u_x)_x+(\kappa u_y)_y = f(x,y)$ in the domain $(-1,1)\times(-1,1)$ with a flower shaped interface given by the following equation in the polar coordinates
\eqm
  r(\theta)=0.5+0.1\sin(8\theta), \quad  0\leq \theta < 2\pi.
\enm
The coefficient $\kappa$ is a piecewise constant: $\kappa^{-}$ inside the interface and $\kappa^{+}$ outside.
The right hand side $f$, the Dirichlet boundary conditions, and the jump conditions are all determined from  the  exact solution, see also Figure \ref{fig:flower},
\eqm
 u(x,y) = \left\{ \begin{array}{ll}
 \dsp \frac{r^2}{\kappa^{-}} & \textrm{if $(x,y)\in\Omega^{-},$}\\ \eqsp
  \dsp \frac{r^4-0.1\log 2r}{\kappa^{+}} & \textrm{if $(x,y)\in\Omega^{+}$}.
\end{array} \right.
\enm

\vspace{-0.15cm}
\begin{table}[hptb]
\centering \caption{Convergence results of two-grid method for Example~\ref{ex-flower}. Left: $\kappa^-=1$, $\kappa^+=10$; Right:
$\kappa^-=50$, $\kappa^+=1$. } \label{tab:flower}

\vthin

\begin{tabular}{|c||c c|}
\hline
     $N$ &   $\|E \|_{\infty}$ (coarse ) & $\|E \|_{\infty}$ (fine)  \\  \hline
           \multirow{3}{*}{40} &  $2.6514\times 10^{-3}$    & $ 3.5781\times 10^{-4}$   \\ \cline{2-3}
      & $4.6060  \times 10^{-4}$ &  $6.3633\times 10^{-4}$     \\  \cline{2-3}
      &  $4.1335 \times 10^{-5}$   & 6.4903  $ \times 10^{-5}$    \\ \hline
 \multirow{3}{*}{80} &  $ 4.9900 \times 10^{-6}$    &   $ 2.1077 \times 10^{-5}$  \\  \cline{2-3}
      & $ 6.8967 \times 10^{-6}$  &  $2.3883 \times 10^{-6}$    \\  \cline{2-3}
      & $ 2.3110 \times 10^{-7}$  & $  6.7858 \times 10^{-7}$   \\   \hline
    \multirow{3}{*}{160} &  $ 9.8969 \times 10^{-7}$    & $ 2.3783 \times 10^{-6}$ \\  \cline{2-3}
     & $ 3.3377 \times 10^{-7}$  &$  6.7788 \times 10^{-7}$    \\  \cline{2-3}
     & $ 7.6309\times 10^{-8}$  &$  1.7808 \times 10^{-7}$     \\  \hline
     \multirow{3}{*}{320}  & $ 2.8677 \times 10^{-7}$ &$  6.8180 \times 10^{-7}$   \\  \cline{2-3}
     & $ 8.4336 \times 10^{-8}$ &  $ 1.7984 \times 10^{-7}$ \\  \cline{2-3}
      & $ 2.1948 \times 10^{-8}$ &$   4.3751 \times 10^{-8}$ \\ \hline
\hline

\end{tabular}
\hspace{0.5cm} %
\begin{tabular}{|c||c c|}
\hline
  $N$ &   $\|E \|_{\infty}$ (coarse ) & $\|E \|_{\infty}$ (fine) \\  \hline
           \multirow{3}{*}{40} &  $2.9005\times 10^{-3}$    & $ 5.1821 \times 10^{-3}$   \\  \cline{2-3}
      & $ 2.4027  \times 10^{-4}$ &  $4.7331\times 10^{-4}$     \\  \cline{2-3}
      &  $  4.1345 \times 10^{-5}$   &   $ 8.0766\times 10^{-5}$    \\ \hline
 \multirow{3}{*}{80} &  $2.5074  \times 10^{-4}$    &   $ 4.7357 \times 10^{-4}$  \\  \cline{2-3}
      & $ 4.3436 \times 10^{-5}$  &  $8.1008 \times 10^{-5}$    \\  \cline{2-3}
      & $ 2.3557 \times 10^{-6}$  & $  4.9060 \times 10^{-6}$   \\   \hline
    \multirow{3}{*}{160} &  $ 4.5306 \times 10^{-5}$    & $ 8.1051 \times 10^{-5}$ \\  \cline{2-3}
     & $  2.4412 \times 10^{-7}$  &$ 4.9550  \times 10^{-6}$    \\  \cline{2-3}
     & $3.8420 \times 10^{-7}$  &$  1.1772 \times 10^{-6}$     \\  \hline
     \multirow{3}{*}{320}  & $ 2.483 \times 10^{-6}$ &$ 4.9601 \times 10^{-6}$   \\  \cline{2-3}
     & $ 4.1551 \times 10^{-7}$ &  $  1.1829 \times 10^{-6}$ \\  \cline{2-3}
      & $ 1.5622 \times 10^{-7}$ &$  2.7301  \times 10^{-7}$ \\ \hline
 \end{tabular}

\end{table}

\begin{figure}[hbtp]
  \includegraphics[width=0.5\textwidth]{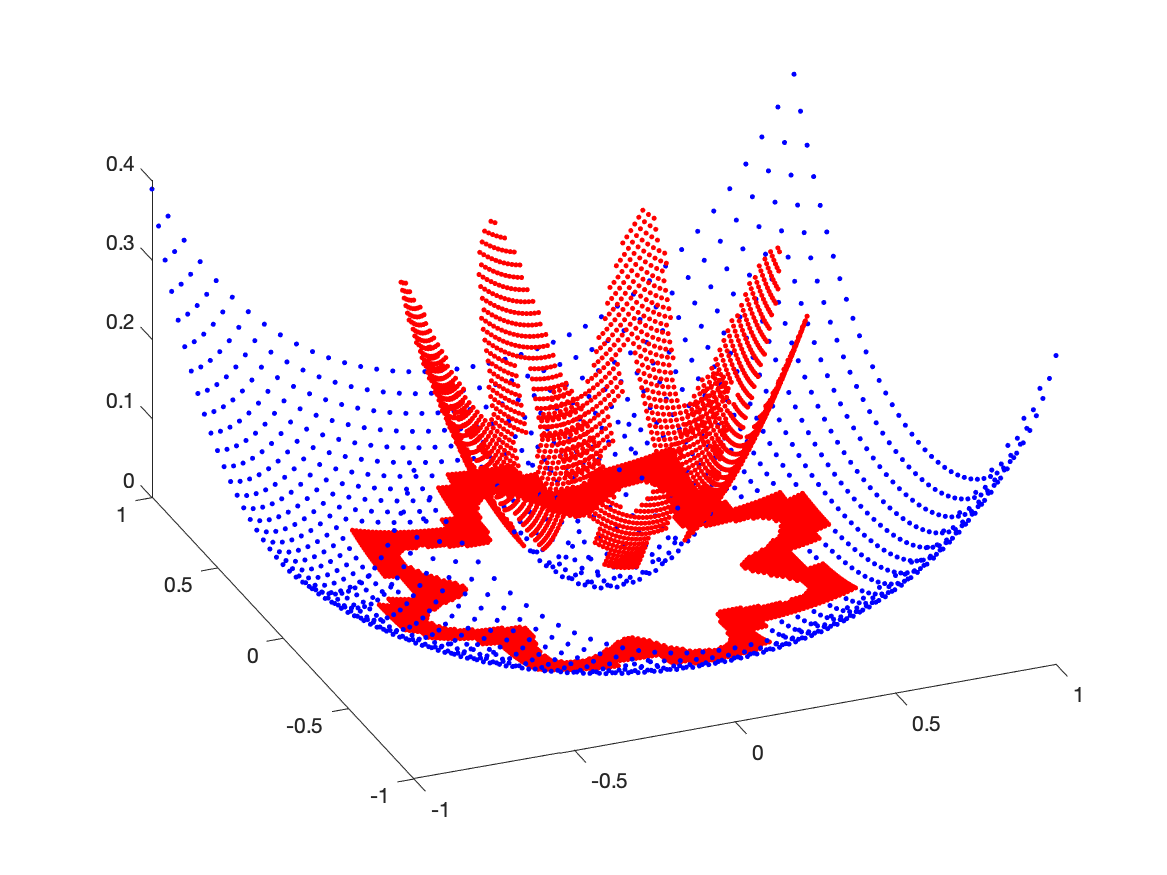}
  \includegraphics[width=0.5\textwidth]{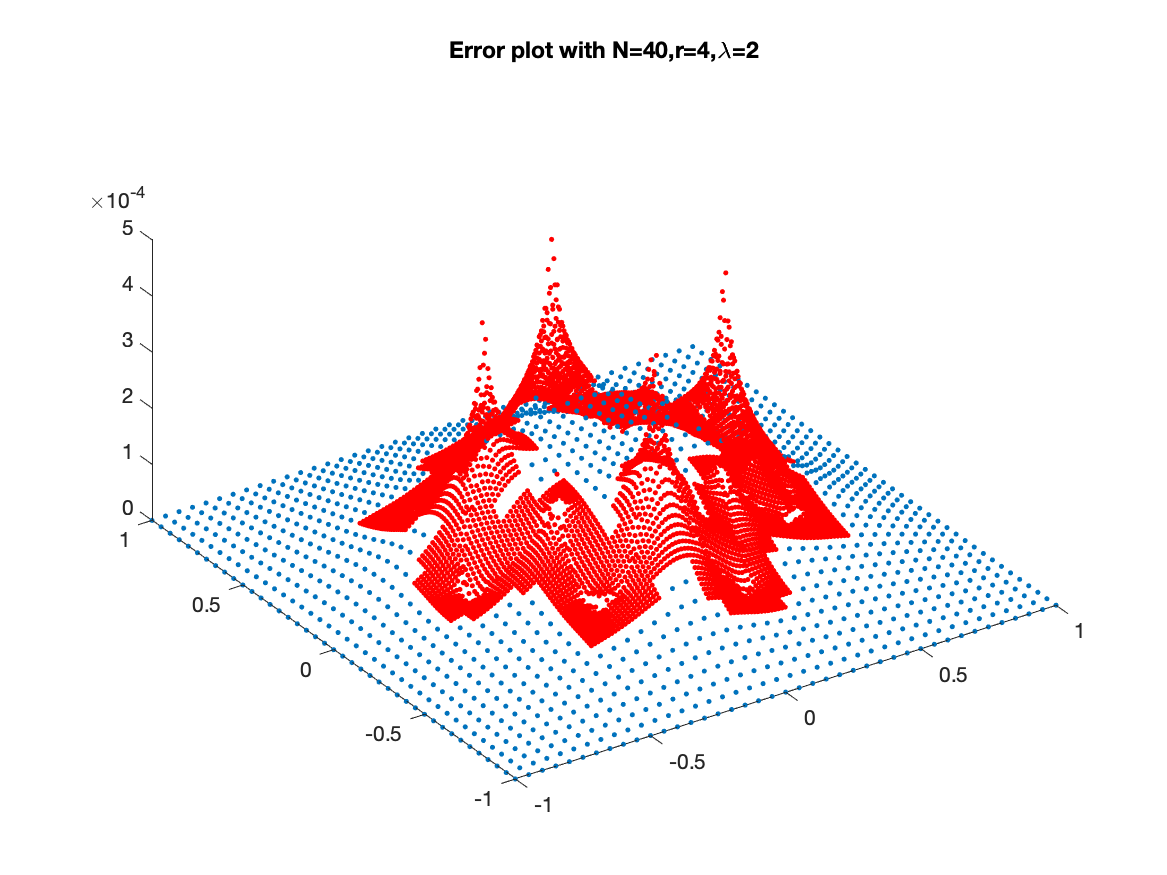}
\caption{The solution plot of the flower example with  $N=40$, $r=2$, and $\lambda=2$. Left plot: $u(x,y)$; Right plot: error plot with $\|E\|_{\infty}=4.7331\times 10^{-4}$.  We use blue dots to represent the values at the coarse grid, and red dots for those at the fine grid.}
\label{fig:flower}
\end{figure}

This is a  general example in which both the coefficients and the solution are discontinuous. The interface is complicated with larger curvature at some part of the interface, see Figure~\ref{fig:flower}. In the figure, we use blue dots to represent the values at the coarse grid, and red dots for those at the fine grid.
In Table~\ref{tab:flower}, we show the numerical results with $r=2,4,8$. We have the same convergence properties as before, that is,  compared with the results using a uniform mesh, we obtain roughly fourth order convergence when  $r=4$.  Note also that the result of the two-grid method  with $r=8$ has a fourth order convergence compared with that obtained with $r=2$.

\subsection{An internal layer  example}

Now we apply  the developed  two-grid method to an internal layer problem in which the solution to the Poisson equation is
\eqm
  u(x,y) = \arctan \left ( \frac{\sqrt{x^2 + y^2 +3/4}-1 }{\epsilon} \right ), \qquad -1 < x, \; y < 1.
\enm
When $\epsilon$ is small, there is an internal layer centered at the circle $x^2+y^2 = 1/4$. The source term below is obtained from Matlab symbolic ToolBox,
\eqml{layer_f}
 && \dsp  f(x,y) = \frac{1}{\epsilon}  \left (  \frac{2}{ (\sigma_3^2 +1 ) \sqrt{\sigma_4} }  -  \frac{2(x^2+y^2)}{\sigma_1} -  \frac{200(x^2+y^2)\sigma_3}{\sigma_2}  \right ), \; \; \mbox{with} \\ \eqsp
  &&  \!\!\! \dsp \sigma_1 = (\sigma_3^2 +1 )\sigma_4^{3/2}, \; \sigma_2 =  (\sigma_3^2 +1 )^2 \sigma_4, \;
\sigma_3 = \frac{1}{\epsilon}  \left (  \sqrt{\sigma_4} -1 \right), \; \sigma_4 = x^2 + y^2 + \frac{3}{4}.
\enml
The solution and the source term are displayed in Figure~\ref{fig:layer} at the coarse grid (blue dots) and the fine grid (red dots) with
$\epsilon=0.01$, $N=40$, $r=4$, and $\lambda=2$.  An internal boundary layer centered along the circle $x^2+y^2 = 1/4$ can be seen clearly.

\begin{figure}[hbtp]
\centering
  \includegraphics[width=0.45\textwidth]{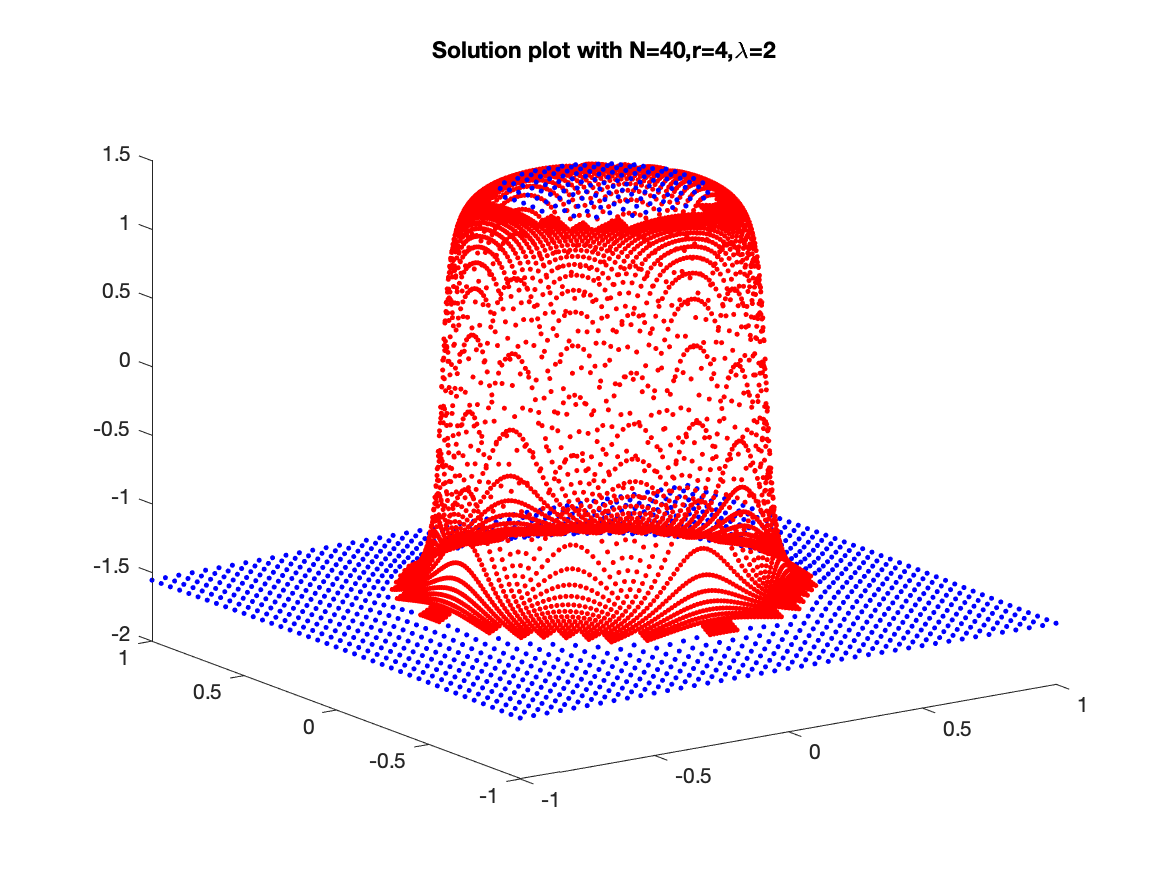}
  \includegraphics[width=0.45\textwidth]{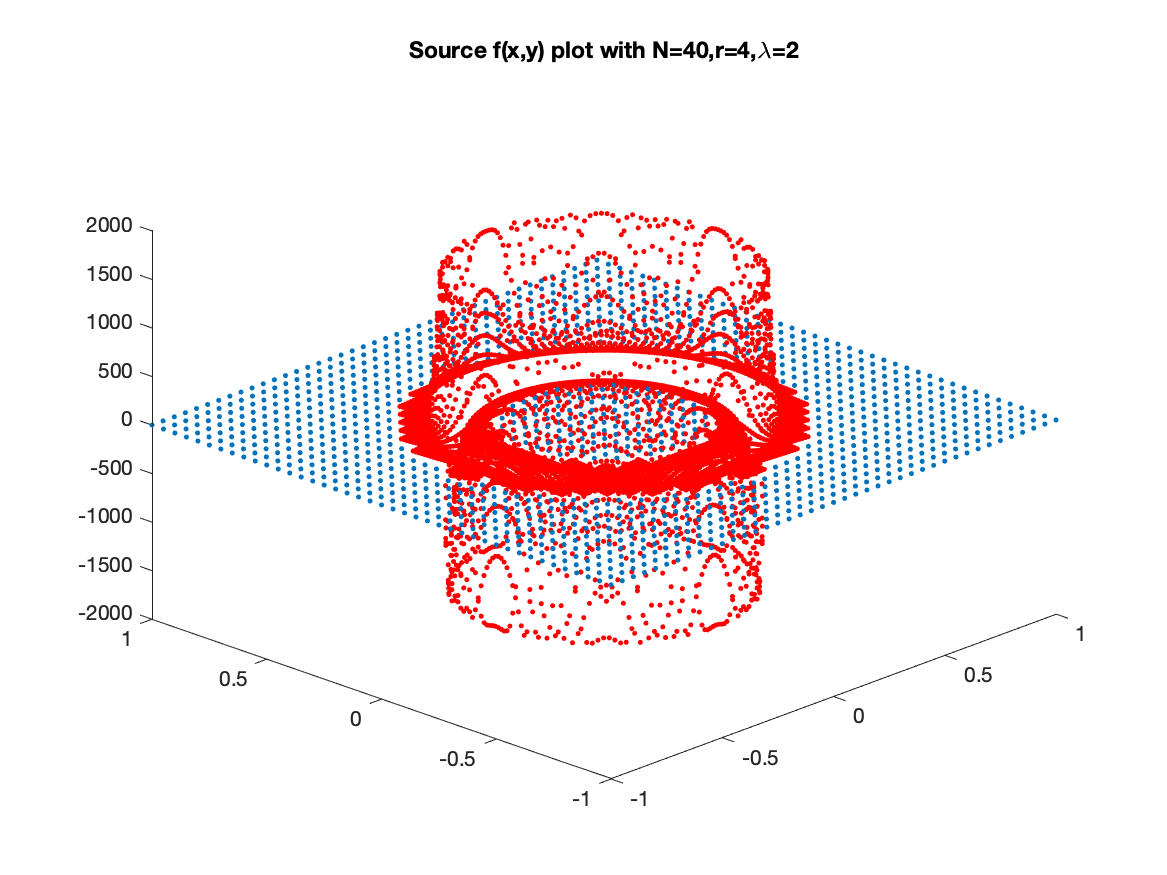}
\caption{The solution plot of the internal layer example with $\epsilon=0.01$, $N=40$, $r=4$, and $\lambda=2$. Left plot: $-u(x,y)$; Right plot: $f(x,y)$ which has large magnitude and variations near the internal layer around $x^2+y^2 = 1/4$. We use blue dots to represent the values at the coarse grid, and red dots for those at the fine grid.}
\label{fig:layer}
\end{figure}

For this problem, we can apply the standard fourth order compact scheme both in the fine and coarse meshes since there is no interface which does not capture the internal layer well unless we use an unfordable fine mesh.
The developed two-grid method worked well for the internal layer problem.
In Table~\ref{tab:layer}, we show some numerical methods with different $N$ and $r$. In this example, the largest errors have mostly appeared at hanging nodes since they have the largest local truncation errors.
We have the same convergence properties as before, that is, from the uniform grid $(r=1)$ to that obtained from $r=4$, or from $r=2$ to $r=8$, we obtain roughly fourth order convergence. The increased computational complexity only happened in the refined tube.

\begin{table}[hptb]
\centering \caption{Convergence results of two-grid method for the internal layer example with $\epsilon=0.01$, and $r=2,4,8$. } \label{tab:layer}

\vthin
\vthin

\begin{tabular}{|c||c c|}
\hline
     $N$ &   $\|E \|_{\infty}$ (coarse ) & $\|E \|_{\infty}$ (fine) \\  \hline
           \multirow{3}{*}{20} &  $ 2.7557$    & $3.8289 $   \\ \cline{2-3}
      & $3.0755  \times 10^{-1}$ &  $4.9217 \times 10^{-1}$     \\  \cline{2-3}
      &  $2.8506 \times 10^{-3}$   &   $ 8.0663 \times 10^{-3}$    \\ \hline
 \multirow{3}{*}{40} &  $ 3.3413  \times 10^{-1}$    &   $ 4.9067 \times 10^{-1}$  \\  \cline{2-3}
      & $ 5.3367\times 10^{-3}$  &  $1.0304 \times 10^{-2}$    \\  \cline{2-3}
      & $ 3.4447 \times 10^{-5}$  & $  2.0639  \times 10^{-4}$   \\   \hline
    \multirow{3}{*}{80} &  $  7.4351 \times 10^{-3}$    & $ 7.4351 \times 10^{-2}$ \\  \cline{2-3}
     & $  1.3716 \times 10^{-4}$  &$   2.5445  \times 10^{-4}$    \\  \cline{2-3}
     & $  4.4450 \times 10^{-6}$  &$  1.4286  \times 10^{-5}$     \\  \hline
     \multirow{3}{*}{160}  & $ 3.9407  \times 10^{-3}$ &$  4.5279  \times 10^{-3}$   \\  \cline{2-3}
     & $ 1.4769  \times 10^{-4}$ &  $ 1.8625  \times 10^{-4}$ \\  \cline{2-3}
      & $ 1.2278   \times10^{-6}$ &$   1.6815  \times 10^{-6}$ \\ \hline
\hline

\end{tabular}
\end{table}

In Figure~\ref{fig:layer}, we plot the computed solution (left) and the source term (right) using blue dots to represent the values at the coarse grid, and red dots for those at the fine grid. The internal layer is caused from large variations in the source term near the center of the internal layer $x^2 + y^2 = 1/4$.

\section{Conclusions and Acknowledgements}

In this paper, we have developed two-grid methods for some elliptic interface and internal layer problems in one and two dimensions. The purpose is focused on the accuracy of the solutions near the interface or internal layers so that the accuracy is comparable with the solution in the coarse grid obtained from fourth order compact finite difference schemes. New high order finite difference schemes have been developed at boarder grid points neighboring two grids for one-dimensional and two-dimensional problems, and at hanging nodes for two dimensional problems.
The developed two-grid methods preserve the discrete maximum principle from the sign property, which leads to the convergence of the finite difference methods.

We would like to thank Dr. Danfu Han of Hangzhou Normal University for the internal layer example.
Zhilin Li was partially supported by a Simons grant 633724. K. Pan is supported by Science Challenge Project (No. TZ2016002), the National Natural Science Foundation of China (No. 41874086), the Excellent Youth Foundation of Hunan Province of China (No. 2018JJ1042).

\parskip= 0.0cm
\renewcommand{\baselinestretch}{1.0}

\bibliographystyle{amsplain}
\bibliography{twogrid}

\end{document}